\documentclass[a4paper,12pt,reqno]{amsart} 

\makeatletter
\tagsleft@false 
\makeatother

\usepackage{xcolor}

\usepackage{comment}

\usepackage{tikz}
\usepackage{pgfplots}
\pgfplotsset{compat=1.18}

\usepackage{amsmath}
\usepackage{amssymb}
\usepackage{mathrsfs}
\usepackage{ifthen}
\usepackage{graphicx}
\usepackage[T1]{fontenc}

\numberwithin{equation}{section}

\newtheorem*{theorem*}{Theorem}
\newtheorem{thm}{Theorem}[section]
\newtheorem{cor}[equation]{Corollary}
\newtheorem{lem}[equation]{Lemma}

\theoremstyle{definition}

\newtheorem{prob}[equation]{Problem}

\newcounter{minutes}
\divide\time by 60
\newcounter{hours}
\multiply\time by 60
\addtocounter{minutes}{-\time}

\newcounter{own}
\def\theown{\thesection.\arabic{own}}

\newenvironment{pf}[1][]{%
	\vskip 3mm
	\noindent
	\ifthenelse{\equal{#1}{}}%
	{{\slshape Proof. }}%
	{{\slshape #1.} }%
}%
{\qed\bigskip}

\newcounter{alphabet}

\def\be{\begin{equation}}
	\def\ee{\end{equation}}

\newcommand{\bee}{\begin{enumerate}}
	\newcommand{\eee}{\end{enumerate}}

\newcommand{\blem}{\begin{lem}}
	\newcommand{\elem}{\end{lem}}
\newcommand{\bthm}{\begin{thm}}
	\newcommand{\ethm}{\end{thm}}
\newcommand{\bcor}{\begin{cor}}
	\newcommand{\ecor}{\end{cor}}
\newcommand{\beg}{\begin{examp}}    
	\newcommand{\eeg}{\end{examp}}
\newcommand{\begs}{\begin{examples}} 
	\newcommand{\eegs}{\end{examples}}
\newcommand{\bdefe}{\begin{defin}}   
	\newcommand{\edefe}{\end{defin}}
\newcommand{\bprob}{\begin{prob}}
	\newcommand{\eprob}{\end{prob}}
\newcommand{\bei}{\begin{itemize}}
	\newcommand{\eei}{\end{itemize}}

\allowdisplaybreaks

\usepackage{verbatim}
\begin{document}
	
	\title{{Sharp Estimates of Hankel Determinants for certain  classes of convex univalent functions
	}}
	
	\author{Vasudevarao Allu}
	\address{Vasudevarao Allu,
		Department of Mathematics,
		School of Basic Sciences,
		Indian Institute of Technology Bhubaneswar,
		Bhubaneswar-752050, Odisha, India.}
	\email{avrao@iitbbs.ac.in}
	
	\author{Shobhit Kumar}
	\address{Shobhit Kumar,
		Department of Mathematics,
		School of Basic Sciences,
		Indian Institute of Technology Bhubaneswar,
		Bhubaneswar-752050, Odisha, India.}
	\email{a21ma09007@iitbbs.ac.in}
	
	
	
	\def\thefootnote{}
	\footnotetext{
		{ }
	}
	\makeatletter\def\thefootnote{\@arabic\c@footnote}\makeatother
	\begin{abstract}
		Let  $\mathcal{A}$ denote the class of  analytic functions $f$ such that $f(0)=0$ and $f'(0)=1$ in the unit disk $\mathbb{D}:=\{z \in \mathbb{C}: |z|<1\}.$
		We examine the properties of the class $\mathcal{C}(\varphi)$ defined as $\mathcal{C}(\varphi) := \left\{ f \in \mathcal{A} : 1+zf''(z)/f'(z) \prec \varphi(z):=1+z+ m/n\, \,  z^2, \text{ with } 2m \le n,\text{ for } m,  n \in \mathbb{N} \right\},$ and compute the sharp second and third Hankel determinants for the functions in $\mathcal{C}(\varphi)$. Furthermore, we determine the extremal functions for the sharp estimates of the Hankel determinants.
	\end{abstract}
	\maketitle
	\markboth{V. Allu and S. Kumar}{Convex Ma--Minda subclass of $\varphi(z)$}
	\setcounter{page}{1}
	
	\section{Introduction}
	
	Let $\mathcal{H}$ denote the class of functions analytic in the open unit disk
	$
	\mathbb{D}:=\{z\in\mathbb{C}:|z|<1\}.
	$
	We write $\mathcal{A}$ for the normalized subclass of $\mathcal{H}$ consisting of functions $f$ satisfying
	$
	f(0)=0
	$
	and
	$
	f'(0)=1.
	$
	Accordingly, each function $f\in\mathcal{A}$ admits the expansion
	\[
	f(z)=z+\sum_{n=2}^{\infty}a_n z^n, \qquad z\in\mathbb{D}.
	\]
	Let $\mathcal{S}$ denote the family of all univalent functions in $\mathcal{A}$. This class occupies a central position in geometric function theory, since several major subclasses, including the starlike and convex families, are defined through geometric properties of the image domain.
	
	A domain $\Omega\subset\mathbb{C}$ is called starlike with respect to a point $\zeta_{0}\in \Omega$ if, for every $\zeta\in \Omega$, the whole line segment joining $\zeta_{0}$ to $\zeta$ is contained in $\Omega$. That is, equivalently,
	\[
	(1-s)\zeta_{0}+s\zeta \in \Omega
	\qquad \text{for all } \zeta\in \Omega \text{ and all } s\in[0,1].
	\]
	A function $f\in\mathcal{S}$ is said to be in the class $\mathcal{S}^*$ whenever the image domain $f(\mathbb{D})$ is starlike with respect to the origin. In analytic terms, this is equivalent to 
	\[
	f\in\mathcal{S}^*
	\quad \text{if, and only if,} \quad
	\operatorname{Re}\left(\frac{\zeta f'(\zeta)}{f(\zeta)}\right)>0,
	\qquad \zeta\in\mathbb{D}.
	\]
	A domain $\Omega\subset\mathbb{C}$ is said to be convex if, whenever $\zeta_{1},\zeta_{2}\in \Omega$, the entire line segment connecting these two points remains inside $\Omega$. In other words,
	\[
	(1-s)\zeta_{1}+s\zeta_{2}\in \Omega
	\qquad \text{for all } \zeta_{1},\zeta_{2}\in \Omega \text{ and all } s\in[0,1].
	\]
	A function $f\in\mathcal{S}$ is called convex in $\mathbb{D}$ whenever the image domain $f(\mathbb{D})$ is convex. This class is denoted by $\mathcal{C}$, and it admits the analytic characterization
	\[
	f\in\mathcal{C}
	\quad \text{if, and only if,} \quad
	\operatorname{Re}\left(1+\frac{\zeta f''(\zeta)}{f'(\zeta)}\right)>0,
	\qquad \zeta\in\mathbb{D}.
	\]
	The classes $\mathcal{S}^*$ and $\mathcal{C}$ have been studied intensively and form a basic part of the modern theory of univalent functions; see Goodman \cite{Goodman1983}.
	
	An important unifying generalization of these classes was formulated by Ma and Minda \cite{MaMinda1992}. Let $\varphi$ be analytic in $\mathbb{D}$ and satisfy
	$
	\operatorname{Re}\varphi(\zeta)>0$
	for $\zeta\in\mathbb{D},$
	together with the normalization conditions
	$
	\varphi(0)=1, 
	\varphi'(0)>0.
	$
	Assume also that the image region $\varphi(\mathbb{D})$ is symmetric with respect to the real axis and starlike with respect to the point $1$. Then the corresponding Ma--Minda  convex class is given by
	\[
	\mathcal{C}(\varphi):=
	\left\{
	f\in\mathcal{A}:
	1+\frac{\zeta f''(\zeta)}{f'(\zeta)}\prec \varphi(\zeta),
	\ \zeta\in\mathbb{D}
	\right\}.
	\]
	This formulation provides a common setting for many previously studied subclasses and has proved useful in coefficient questions, extremal problems, inclusion relations, and radius problems; (see \cite{MaMinda1992}).
	
	Coefficient functionals also form an important part of the theory of analytic and univalent functions. Among them, Hankel determinants, introduced by Pommerenke \cite{Pommerenke1966}, have been studied extensively for many subclasses of analytic functions; for example, (see \cite{KwonLeckoSim2018, LiberaZlotkiewicz1982,ProkhorovSzynal1981}). In the present work, we investigate the second and third Hankel determinants for the class $\mathcal{C}(\varphi)$, and their precise definitions together with the required auxiliary results are recalled in Section~3.
	
	In the present paper we study the Ma--Minda convex class determined by
	$
	\varphi(z)=1+z+m/n \, \, z^2,
	$
	where $m,n\in\mathbb{N}$ and $2m\le n$. After establishing that this choice of $\varphi$ satisfies the required Ma--Minda conditions under the assumption $2m\le n$, we derive sharp estimates for the second Hankel determinant $H_2(2)$ and the third Hankel determinant $H_3(1)$ for functions in $\mathcal{C}(\varphi)$. We also identify the corresponding extremal functions.

	\section{Certain properties of the function $\varphi(z)$}
	First, we show that the function satisfies all the properties of a Ma-Minda function. It is easy to see that $\varphi(0)=1$ and $\varphi'(0)=1.$ Furthermore,  $\varphi(z)$ is symmetric with respect to the real axis.\\[2mm]
	Next, we show that $\varphi(z)$ is univalent in the unit disk $\mathbb{D}$ whenever $2m \le n.$
	Taking, 
	$
	a={m}/{n}>0,
	$
	then
	\[
	\varphi(z)=1+z+az^2.
	\]
	To determine when \(\varphi\) is univalent in the unit disk \(\mathbb D\). Let \(z_1,z_2\in\mathbb D\). We compute
	\[
	\varphi(z_1)-\varphi(z_2)
	=(z_1-z_2)+a(z_1^2-z_2^2)
	=(z_1-z_2)\bigl(1+a(z_1+z_2)\bigr).
	\]
	If $\varphi(z_1)=\varphi(z_2)$ and $z_1\ne z_2$, then, we have 
	$
	1+a(z_1+z_2)=0.
	$
	Thus, we have 
	\[
	1+a(z_1+z_2)=0,
	\]
	that is,
	\[
	z_1+z_2=-\frac1a.
	\]
	Thus \(\varphi\) fails to be univalent precisely when there exist distinct points
	\(z_1,z_2\in\mathbb D\) such that
	$
	z_1+z_2=-1/a.
	$
	Now, for \(z_1,z_2\in\mathbb D\), we have
	\[
	|z_1+z_2|\le |z_1|+|z_2|<2.
	\]
	Hence, if such \(z_1,z_2\) exist, then necessarily
	\[
	\left|-\frac1a\right|<2,
	\]
	or equivalently,
	$
	1/a<2,
	$
	which implies
	$
	a>1/2.
	$
	Therefore,  if \(a\le 1/2\), then no such pair \(z_1,z_2\in\mathbb D\) can exist, and therefore
	\(\varphi\) is univalent in \(\mathbb D\).
	Conversely, suppose \(a>1/2\). Then
	\[
	\frac1a<2,
	\qquad\text{so}\qquad
	\left|-\frac1a\right|<2.
	\]
	Let
	$
	c=-{1}/{2a}.
	$
	Then \(|c|<1\). Choose \(\varepsilon>0\) small enough so that
	\[
	z_1=c+\varepsilon,
	\qquad
	z_2=c-\varepsilon
	\]
	both belong to \(\mathbb D\). Clearly \(z_1\ne z_2\), and
	\[
	z_1+z_2=2c=-\frac1a.
	\]
	Hence
	\[
	1+a(z_1+z_2)=1+a\left(-\frac1a\right)=0,
	\]
	and therefore
	\[
	\varphi(z_1)=\varphi(z_2).
	\]
	So \(\varphi\) is not univalent in \(\mathbb D\).
	We conclude that \(\varphi\) is univalent in \(\mathbb D\) if, and only if, 
	$
	a\le 1/2.
	$
	Since \(a={m}/{n}\), this is equivalent to
	$
	{m}/{n}\le 1/2,
	$
	that is,
	$
	2m\le n.
	$\\[2mm]
	Next, we prove that the function $\varphi(z)$ is starlike with respect to the point $\varphi(0)=1$ given that $2m \le n.$
	Taking, 
	$
	a={m}/{n}>0.
	$
	then, we may write $\varphi(z)$ as:
	\[
	\varphi(z)=1+z+az^2.
	\]
	The function \(\varphi\) is starlike with respect to \(\varphi(0)=1\) if, and only if, 
	\[
	g(z):=\varphi(z)-1=z+az^2
	\]
	is starlike with respect to the origin.
	Now,
	\[
	g'(z)=1+2az,
	\]
	and therefore
	\begin{align*}
		\frac{z g'(z)}{g(z)}
		=\frac{z(1+2az)}{z(1+az)}
		=\frac{1+2az}{1+az}.
	\end{align*}
	By the standard characterization of starlike functions, \(g\) is starlike in \(\mathbb D\) if, and only if, 
	\[
	\operatorname{Re}\left(\frac{z g'(z)}{g(z)}\right)>0,
	\qquad z\in\mathbb D,
	\]
	that is,
	\[
	\operatorname{Re}\left(\frac{1+2az}{1+az}\right)>0,
	\qquad z\in\mathbb D.
	\]
	First, we  suppose that
	$
	a\le 1/2,
	$
	let $w=az$. Then \(|w|<a\le 1/2\), and
	\[
	\frac{1+2az}{1+az}
	=\frac{1+2w}{1+w}
	=1+\frac{w}{1+w}.
	\]
	Also,
	\[
	\left|\frac{w}{1+w}\right|
	\le \frac{|w|}{1-|w|}
	<1.
	\]
	Hence
	\[
	\operatorname{Re}\left(1+\frac{w}{1+w}\right)>0,
	\]
	which gives
	\[
	\operatorname{Re}\left(\frac{1+2az}{1+az}\right)>0,
	\qquad z\in\mathbb D.
	\]
	Thus \(g\) is starlike in \(\mathbb D\), and so \(\varphi\) is starlike with respect to \(1\).
	Conversely, suppose that
	$
	a>1/2.
	$
	Choose \(r\) such that
	\[
	\frac{1}{2a}<r<\min\left\{1,\frac{1}{a}\right\},
	\]
	and set \(z=-r\). Then \(z\in\mathbb D\), and
	\[
	\frac{1+2az}{1+az}
	=\frac{1-2ar}{1-ar}.
	\]
	Since \(r>{1}/{2a}\), we have \(1-2ar<0\), and since \(r<1/a\), we have \(1-ar>0\). Therefore
	\[
	\frac{1-2ar}{1-ar}<0.
	\]
	Hence,
	\[
	\operatorname{Re}\left(\frac{1+2az}{1+az}\right)<0
	\]
	for this choice of \(z\), so \(g\) is not starlike in \(\mathbb D\). Consequently, \(\varphi\) is not starlike with respect to \(1\).
	Therefore \(\varphi\) is starlike with respect to \(\varphi(0)=1\) if, and only if, 
	$
	a\le 1/2.
	$
	Since \(a={m}/{n}\), we get $2m \le n.$\\[2mm]
	Now, we prove that the $\operatorname{Re}\varphi(z)$ remains positive in the unit disk $\mathbb{D}.$
	Again taking 
	$
	a={m}/{n}>0.
	$
	Then
	$
	\varphi(z)=1+z+az^2.
	$
	We want to determine for which \(a>0\), we have 
	$
	\operatorname{Re} \varphi(z)>0
	$ in $\mathbb D.$\\
	Consider the harmonic function
	$
	u(z)=\operatorname{Re} \varphi(z).
	$
	Since \(u\) is harmonic in \(\mathbb D\) and continuous on \(\overline{\mathbb D}\), it is enough to study its boundary values on \(|z|=1\).
	Let \(z=e^{it}\), where \(t\in\mathbb R\). Then
	\[
	u(e^{it})=\operatorname{Re}\, \bigl(1+e^{it}+ae^{2it}\bigr)
	=1+\cos t+a\cos 2t.
	\]
	Taking 
	$
	x=\cos t\in[-1,1],
	$
	we obtain
	\[
	u(e^{it})=1+x+a(2x^2-1)=2ax^2+x+1-a.
	\]
	Thus it suffices to determine when the quadratic polynomial
	\[
	q(x)=2ax^2+x+1-a
	\]
	is non-negative for all \(x\in[-1,1]\).
	Now, taking the derivative of $q(x)$ with respect to $x,$
	\[
	q'(x)=4ax+1.
	\]
	Hence the critical point is
	$
	x_0=-{1}/{4a}.
	$
	Thus, we discuss the following two  cases.\\[2mm]
	Case 1:  If \(0<a\le 1/4\), then
	\[
	x_0=-{1}/{4a}\le -1.
	\]
	Since \(q(x)\) is convex, its minimum on \([-1,1]\) occurs at \(x=-1\). Therefore
	\[
	\min_{x\in[-1,1]} q(x)=q(-1)=2a-1+1-a=a>0.
	\]
	Hence
	$
	u(e^{it})>0$ for 
	$|z|=1$.\\[2mm]
	Case 2:  Now suppose \(a>1/4\), then
	\[
	x_0=-\frac{1}{4a}\in(-1,0),
	\]
	so the minimum of \(q\) on \([-1,1]\) occurs at \(x=x_0\). A simple computation gives
	\[
	q\!\left(-\frac{1}{4a}\right)
	=2a\left(\frac{1}{16a^2}\right)-\frac{1}{4a}+1-a
	=1-a-\frac{1}{8a}.
	\]
	Thus \(q(x)\ge 0\) on \([-1,1]\) if, and only if, 
	\[
	1-a-\frac{1}{8a}\ge 0.
	\]
	Since \(a>0\), this is equivalent to
	$
	8a-8a^2-1\ge 0,
	$
	The roots of the quadratic equation
	$
	8a^2-8a+1=0
	$
	are
	\[
	a=\frac{2\pm\sqrt2}{4}.
	\]
	Therefore,
	$
	8a^2-8a+1\le 0, 
	$
	Thus, we have
	\[
	\frac{2-\sqrt2}{4}\le a\le \frac{2+\sqrt2}{4}.
	\]
	Since in the present case \(a>1/4\), only the upper bound is relevant, and we obtain
	\[
	a\le \frac{2+\sqrt2}{4}.
	\]
	Combining the two cases, we conclude that
	$
	u(e^{it})\ge 0$  for $ |z|=1
	$
	if, and only if, 
	\[
	0<a\le \frac{2+\sqrt2}{4}.
	\]
	Because \(u\) is harmonic and not identically zero, the boundary condition \(u\ge 0\) implies
	$
	u(z)>0 
	$
	in $\mathbb{D}.$
	Hence,
	$
	\operatorname{Re} \varphi(z)>0
	$
	if, and only if, 
	\[
	a\le \frac{2+\sqrt2}{4}.
	\]
	Since \(a={m}/{n}\), this is equivalent to
	\[
	\frac{m}{n}\le \frac{2+\sqrt2}{4}.
	\]
	Therefore, from the preceding observations, we conclude that \(\varphi\) satisfies all the Ma--Minda conditions whenever
	$
	2m\le n,$
	for $m,n\in\mathbb N$. Hence, we discuss the properties of the class $\mathcal{C}(\varphi),$ when $2m \le n,$ where $m,n \in \mathbb{N}.$

	\section{Hankel determinants}
	
	The notion of Hankel determinants for functions in the class $\mathcal{S}$ was introduced by Pommerenke \cite{Pommerenke1966}. If \(f\in\mathcal{A}\) has the Taylor expansion
	\[
	f(z)=z+\sum_{n=2}^{\infty}a_n z^n,
	\]
	then the \(q\)th Hankel determinant of \(f\) is given by
	\begin{equation}\label{eq:3.1}
		H_q(n)=
		\begin{vmatrix}
			a_n & a_{n+1} & \cdots & a_{n+q-1} \\
			a_{n+1} & a_{n+2} & \cdots & a_{n+q} \\
			\vdots & \vdots & \ddots & \vdots \\
			a_{n+q-1} & a_{n+q} & \cdots & a_{n+2q-2}
		\end{vmatrix},
		\qquad q,n\in\mathbb{N}.
	\end{equation}
	For particular choices of $n$ and $q$, this determinant reduces to several well-known coefficient functionals. For example, the second Hankel determinant is defined as 
	\begin{align}\label{eq:3.2}
		H_2(2)=a_2a_4-a_3^{\,2}.
	\end{align}
	Since functions in the class  $\mathcal{A}$ are normalized, we adopt the convention $a_1:=1$. Thus the third Hankel determinant is
	\begin{equation}\label{eq:3.3}
		H_3(1)=
		a_3(a_2a_4-a_3^2)-a_4(a_4-a_2a_3)+a_5(a_3-a_2^2).
	\end{equation}
	Let $\mathcal{P}$ denote the class of analytic functions $p$ in $\mathbb{D}$ of the form
	\begin{equation}\label{eq:3.4}
		p(z)=1+\sum_{n=1}^{\infty} p_n z^n
	\end{equation}
	with $\operatorname{Re} p(z)>0$ in $\mathbb{D}$. This class is known as Carath\'eodory class.
	We recall the following well-known result of Choi \textit{et al.} \cite{ChoiKimSugawa2007}, which is essential for the proof of our main theorem.
	
	\begin{lem}\label{lemma2} {\cite{ChoiKimSugawa2007}}
		Let $A,B,C\in\mathbb R$ and define
		\[
		Y(A,B,C):=\max_{z\in\overline{\mathbb D}}\Bigl(|A+Bz+Cz^{2}|+1-|z|^{2}\Bigr).
		\]
		\begin{itemize}
			\item[(i)] If $AC\ge 0$, then
			\[
			Y(A,B,C)=
			\begin{cases}
				|A|+|B|+|C|, & \text{if } |B|\ge 2(1-|C|),\\[6pt]
				1+|A|+\dfrac{B^{2}}{4(1-|C|)}, & \text{if } |B|<2(1-|C|).
			\end{cases}
			\]
			\item[(ii)] If $AC<0$, then
			\[
			Y(A,B,C)=
			\begin{cases}
				1-|A|+\dfrac{B^{2}}{4(1-|C|)}, & \text{if } -4AC(C^{2}-1)\le B^{2}\ \text{and } |B|<2(1-|C|),\\[8pt]
				1+|A|+\dfrac{B^{2}}{4(1+|C|)}, & \text{if } B^{2}<\min\!\bigl\{\,4(1+|C|)^{2},-4AC(C^{2}-1)\bigr\},\\[8pt]
				R(A,B,C), & \text{otherwise,}
			\end{cases}
			\]
			where
			\[
			R(A,B,C)=
			\begin{cases}
				|A|+|B|+|C|, & \text{if } |C|\bigl(|B|+4|A|\bigr)\le |AB|,\\[6pt]
				-|A|+|B|+|C|, & \text{if } |AB|\le |C|\bigl(|B|-4|A|\bigr),\\[10pt]
				\bigl(|A|+|C|\bigr)\sqrt{\,1-\dfrac{B^{2}}{4AC}\,}, & \text{otherwise.}
			\end{cases}
			\]
		\end{itemize}
	\end{lem}
	
	We shall also use the following coefficient representation of Schwarz function.  The next well-known lemma by Prokhorov and Szynal \cite{ProkhorovSzynal1981} will be crucial for our result.
	
	\begin{lem}\label{lemma3}
		Let
		$
		w(z)=\sum_{n=1}^{\infty}c_n z^n
		$
		be a Schwarz function, that is, $w$ is analytic in $\mathbb D$, $w(0)=0$, and $|w(z)|<1$ for all $z\in\mathbb D$. If $c_1\ge 0$, then there exist complex numbers $\gamma,\eta,\rho$ with
		$
		|\gamma|\le 1,\ |\eta|\le 1,\ |\rho|\le 1,
		$
		such that
		\begin{align*}
			c_2 &= (1-c_1^2)\gamma,\\[2mm]
			c_3 &= (1-c_1^2)\bigl(\eta(1-|\gamma|^2)-c_1\gamma^2\bigr),\\[2mm]
			c_4 &= (1-c_1^2)\Bigl(
			c_1^2\gamma^3
			-(1-|\gamma|^2)\bigl(2c_1\gamma\eta+\overline{\gamma}\eta^2\bigr)
			+(1-|\gamma|^2)(1-|\eta|^2)\rho
			\Bigr).
		\end{align*}
	\end{lem}
	\begin{thm}
		Let
		$
		f\in\mathcal{C}(\varphi).
		$
		Then the second Hankel determinant satisfies
		\[
		|H_2(2)|=|a_2a_4-a_3^2|
		\le
		\begin{cases}
			\dfrac{1}{36}, & 0\le t\le \dfrac14,\\[2mm]
			\dfrac{1}{144}\left(4+\dfrac{(4t-1)^2}{8+20t-16t^2}\right), & \dfrac14\le t\le \dfrac12.
		\end{cases}
		\]
		Moreover, the estimate is sharp.
	\end{thm}
	
	\begin{proof}
		Let $f\in \mathcal{C}(\varphi)$ be given by
		\begin{align}\label{eq:3.7}
			1+z \frac{f''(z)}{f'(z)} \prec \varphi(z), \quad z\in \mathbb{D}.
		\end{align}
		By the subordination relation, there exists a Schwarz function
		\[
		w(z)=\sum_{n=1}^{\infty}c_n z^n
		=c_1z+c_2z^2+c_3z^3+c_4z^4+\cdots
		\]
		such that
		\begin{align}\label{eq:3.8}
			1+\frac{z f''(z)}{f'(z)}=1+w(z)+t\,w(z)^2.
		\end{align}
		Since \(f\in\mathcal A\), we have
		\[
		f(z)=z+\sum_{n=2}^{\infty} a_n z^n
		\]
		Writing \eqref{eq:3.8} in series form and comparing coefficients, we obtain
		\[
		a_2=\frac{c_1}{2},\qquad
		a_3=\frac{(1+t)c_1^2+c_2}{6},\qquad
		a_4=\frac{(1+3t)c_1^3+(3+4t)c_1c_2+2c_3}{24}.
		\]
		Substituting these into \eqref{eq:3.2}, we obtain
		\begin{align}\label{eq:3.9}
			H_2(2)
			&=a_2a_4-a_3^2\notag\\
			&=\frac{(4t+1)c_1^2c_2+6c_1c_3-4c_2^2-(4t^2-t+1)c_1^4}{144}.
		\end{align}
		Now, since the class \(\mathcal C(\varphi)\) is invariant under rotations and \(|H_2(2)|\) is rotation-invariant, we may assume without loss of generality that \(c_1\ge 0\). Write
		\[
		c_1=x,\qquad 0\le x\le 1.
		\]
		Using Lemma \ref{lemma3} in \eqref{eq:3.9}, we obtain
		\begin{align}\label{eq:3.10}
			144\,H_2(2)
			&=-(4t^2-t+1)x^4+(4t+1)x^2(1-x^2)\gamma \notag\\
			&\quad -2(1-x^2)(2+x^2)\gamma^2
			+6x(1-x^2)(1-|\gamma|^2)\eta.
		\end{align}
		Taking absolute values in \eqref{eq:3.10}, and setting \(y=|\gamma|\), we get
		\begin{align}\label{eq:3.11}
			144\,|H_2(2)|
			&\le (4t^2-t+1)x^4+(4t+1)x^2(1-x^2)y \notag\\
			&\quad +2(1-x^2)(2+x^2)y^2
			+6x(1-x^2)(1-y^2).
		\end{align}
		We now simplify the last two terms:
		\begin{align*}
			&2(1-x^2)(2+x^2)y^2+6x(1-x^2)(1-y^2)\\
			&=(1-x^2)\bigl(2(2+x^2)y^2+6x(1-y^2)\bigr)\\
			&=(1-x^2)\bigl(6x+(4+2x^2-6x)y^2\bigr).
		\end{align*}
		Hence \eqref{eq:3.11} becomes
		\begin{align*}
			144\,|H_2(2)|
			&\le (4t^2-t+1)x^4+(4t+1)x^2(1-x^2)y \\
			&\quad +(1-x^2)(4+2x^2-6x)y^2
			+6x(1-x^2).
		\end{align*}
		For fixed \(x\) and \(t\), define
		\begin{align}\label{eq:3.12}
			F_{1}(y)
			&:=(4t^2-t+1)x^4+(4t+1)x^2(1-x^2)y \notag\\
			&\quad +(1-x^2)(4+2x^2-6x)y^2
			+6x(1-x^2),
		\end{align}
		where \(0\le y\le 1\). Since
		\[
		4+2x^2-6x=2(1-x)(2-x)\ge 0
		\qquad (0\le x\le 1),
		\]
		the coefficients of \(y\) and \(y^2\) in \(F_1(y)\) are non-negative. Indeed,
		\[
		(4t+1)x^2(1-x^2)\ge 0,
		\qquad
		(1-x^2)(4+2x^2-6x)\ge 0.
		\]
		Therefore
		\[
		F_1'(y)
		=
		(4t+1)x^2(1-x^2)
		+
		2(1-x^2)(4+2x^2-6x)y
		\ge 0
		\]
		for all \(y\in[0,1]\). Hence \(F_1\) is increasing on \([0,1]\), and so its maximum occurs at \(y=1\). Consequently,
		\begin{align*}
			144\,|H_2(2)|
			&\le F_1(1)\\
			&=(4t^2-t+1)x^4+(4t+1)x^2(1-x^2)\\
			&\quad +(1-x^2)(4+2x^2-6x)+6x(1-x^2).
		\end{align*}
		Combining the last two terms, we get
		\[
		(1-x^2)(4+2x^2-6x)+6x(1-x^2)=2(1-x^2)(2+x^2).
		\]
		Thus
		\begin{align*}
			144\,|H_2(2)|
			&\le (4t^2-t+1)x^4+(4t+1)x^2(1-x^2)+2(1-x^2)(2+x^2)\\
			&=4+(4t-1)x^2+(4t^2-5t-2)x^4.
		\end{align*}
		Taking, 
		$
		s=x^2,\, 0\le s\le 1,
		$
		we obtain
		\[
		144\,|H_2(2)|\le \psi_t(s),
		\qquad
		\psi_t(s):=4+(4t-1)s+(4t^2-5t-2)s^2.
		\]
		Since
		\[
		4t^2-5t-2<0
		\quad \text{ for }\quad 0\le t\le \frac12,
		\]
		the function \(\psi_t\) is concave on \([0,1]\).\\[2mm]
		If \(0\le t\le 1/4\), then \(4t-1\le 0\), and therefore the maximum of \(\psi_t\) on \([0,1]\) is attained at \(s=0\). Thus
		\[
		\max_{0\le s\le 1}\psi_t(s)=4,
		\]
		which gives
		\[
		|H_2(2)|\le \frac{4}{144}=\frac{1}{36}.
		\]
		If \(1/4\le t\le 1/2\), then the vertex of \(\psi_t\) is
		\[
		s_0=\frac{4t-1}{4+10t-8t^2}.
		\]
		Since \(4t-1\ge 0\) and \(4+10t-8t^2>0\) for \(1/4\le t\le 1/2\), we have \(s_0\ge 0\). Also,
		\[
		1-s_0
		=
		\frac{4+10t-8t^2-(4t-1)}{4+10t-8t^2}
		=
		\frac{5+6t-8t^2}{4+10t-8t^2}>0,
		\]
		for \(1/4\le t\le 1/2\). Hence \(s_0\in[0,1]\). Therefore
		\[
		\max_{0\le s\le 1}\psi_t(s)
		=
		\psi_t(s_0)
		=
		4+\frac{(4t-1)^2}{8+20t-16t^2}.
		\]
		Thus
		\[
		|H_2(2)|
		\le
		\frac{1}{144}\left(4+\frac{(4t-1)^2}{8+20t-16t^2}\right).
		\]
		Combining the two cases, we obtain
		\[
		|H_2(2)|
		\le
		\begin{cases}
			\dfrac{1}{36}, & 0\le t\le \dfrac14,\\[2ex]
			\dfrac{1}{144}\left(4+\dfrac{(4t-1)^2}{8+20t-16t^2}\right), & \dfrac14\le t\le \dfrac12.
		\end{cases}
		\]
		To see sharpness, first note that for \(0\le t\le 1/4\), equality is attained for the function \(f\in\mathcal C(\varphi)\) corresponding to the Schwarz function
		$
		w(z)=z^2.
		$
		In this case, the subordinating relation
		\[
		1+\frac{z f''(z)}{f'(z)}=1+w(z)+t\,w(z)^2
		\]
		which is equivalent to
		\[
		1+\frac{z f''(z)}{f'(z)}=1+z^2+t z^4,
		\]
		so that
		\[
		\frac{f''(z)}{f'(z)}=z+t z^3.
		\]
		Thus, we have 
		\[
		\frac{d}{dz}\bigl(\log f'(z)\bigr)=z+t z^3.
		\]
		Integrating and using \(f'(0)=1\), we obtain
		\[
		f'(z)=\exp\!\left(\frac{z^2}{2}+\frac{t z^4}{4}\right).
		\]
		A further integration, together with the initial condition \(f(0)=0\), gives
		\[
		f(z)=\int_0^z \exp\!\left(\frac{\xi^2}{2}+\frac{t \xi^4}{4}\right)\,d\xi.
		\]
		Expanding,
		\[
		f(z)
		=
		z+\frac{z^3}{6}
		+\frac{1+2t}{40}\,z^5
		+\frac{1+6t}{336}\,z^7+\cdots.
		\]
		In particular,
		\[
		a_2=0,\qquad a_3=\frac16,\qquad a_4=0.
		\]
		Hence
		\[
		H_2(2)=a_2a_4-a_3^2=-\frac{1}{36},
		\]
		and so
		\begin{align*}
			|H_2(2)|=\frac{1}{36}.
		\end{align*}
		For \(1/4\le t\le 1/2\), equality is attained for the function \(f\in\mathcal C(\varphi)\) corresponding to the Schwarz function
		\[
		w_x(z)=z\frac{x-z}{1-xz},
		\qquad
		x^2=\frac{4t-1}{4+10t-8t^2}.
		\]
		Expanding \(w_x\), we get
		\begin{align*}
			w_x(z)
			&=(xz-z^2)\sum_{n=0}^\infty x^n z^n \\
			&=xz+(x^2-1)z^2+(x^3-x)z^3+(x^4-x^2)z^4+\cdots .
		\end{align*}
		Thus
		\[
		c_1=x,\qquad c_2=x^2-1,\qquad c_3=x^3-x,\qquad c_4=x^4-x^2.
		\]
		Using the coefficient formulas,
		\[
		a_2=\frac{c_1}{2},\qquad
		a_3=\frac{(1+t)c_1^2+c_2}{6},\qquad
		a_4=\frac{(1+3t)c_1^3+(3+4t)c_1c_2+2c_3}{24},
		\]
		we obtain
		\begin{align*}
			a_2&=\frac{x}{2},\\[2mm]
			a_3&=\frac{(2+t)x^2-1}{6},\\[2mm]
			a_4&=\frac{x\bigl((6+7t)x^2-(5+4t)\bigr)}{24}.
		\end{align*}
		Hence the corresponding extremal function \(f\) has the expansion
		\begin{align*}
			f(z)
			&=
			z+\frac{x}{2}z^2
			+\frac{(2+t)x^2-1}{6}z^3
			+\frac{x\bigl((6+7t)x^2-(5+4t)\bigr)}{24}z^4
			+\cdots,
		\end{align*}
		where
		\[
		x^2=\frac{4t-1}{4+10t-8t^2}.
		\]
		Moreover,
		\[
		c_2=x^2-1=-(1-x^2)=(1-c_1^2)\gamma,
		\]
		thus \(\gamma=-1\). Hence
		$
		y=|\gamma|=1,$ and
		$1-|\gamma|^2=0$.
		Therefore the \(\eta\)-term in \eqref{eq:3.10} vanishes, and \eqref{eq:3.11} becomes an equality. In fact,
		\begin{align*}
			144\,H_2(2)
			&=-(4t^2-t+1)x^4+(4t+1)x^2(x^2-1)+6x(x^3-x)-4(x^2-1)^2\\
			&=-(4t^2-t+1)x^4-(4t+1)x^2(1-x^2)-2(1-x^2)(2+x^2).
		\end{align*}
		Hence
		\[
		144\,|H_2(2)|
		=
		(4t^2-t+1)x^4+(4t+1)x^2(1-x^2)+2(1-x^2)(2+x^2)
		=
		\psi_t(x^2).
		\]
		Now choosing
		\[
		x^2=\frac{4t-1}{4+10t-8t^2}=s_0,
		\]
		we obtain
		\[
		144\,|H_2(2)|=\psi_t(s_0)
		=
		4+\frac{(4t-1)^2}{8+20t-16t^2}.
		\]
		Therefore
		\[
		|H_2(2)|
		=
		\frac{1}{144}\left(4+\frac{(4t-1)^2}{8+20t-16t^2}\right),
		\]
		and the bound is sharp.
	\end{proof}
	\begin{thm}
		Let 	$
		f\in\mathcal{C}(\varphi),
		$
		then
		$
		|H_3(1)|\le {1}/{144}.
		$
		Moreover, the inequality is sharp.
	\end{thm}
	\begin{pf}
		Let $f\in \mathcal{C}(\varphi)$ be given by 
		\[
		1+z\frac{f''(z)}{f'(z)} \prec \varphi(z), \quad z\in \mathbb{D}.
		\]
		Recall that for the third Hankel determinant , we need to evaluate 
		\begin{align}\label{eq:3.13}
			H_{3}(1)=
			a_3\bigl(a_2a_4-a_3^2\bigr)-a_4\bigl(a_4-a_2a_3\bigr)+a_5\bigl(a_3-a_2^2\bigr),
		\end{align}
		where $f$ has the series representation 
		\[
		f(z)=z+\sum_{n=2}^{\infty} a_{n} z^{n}=z+a_{2}z^2+a_{3}z^3+a_{4}z^4+\cdots.
		\]
		Now comparing the coefficient of \eqref{eq:3.8} we obtain 
		\begin{align}\label{eq:3.14}
			\left.
			\begin{aligned}
				a_2 &= \frac{c_1}{2},\\[4pt]
				a_3 &= \frac{1}{6}\bigl((1+t)c_1^2+c_2\bigr),\\[4pt]
				a_4 &= \frac{1}{24}\bigl((1+3t)c_1^3+(3+4t)c_1c_2+2c_3\bigr),\\[4pt]
				a_5 &= \frac{1}{120}\bigl((1+6t+3t^2)c_1^4+2(3+11t)c_1^2c_2+(3+6t)c_2^2 \\
				&\qquad\quad +4(2+3t)c_1c_3+6c_4\bigr).
			\end{aligned}
			\right\}
		\end{align} 
		Now, putting the values of \eqref{eq:3.14} in \eqref{eq:3.13}, we obtain 
		\begin{align}\label{eq:3.15}
			8640\, H_{3}(1)=&-\bigl(1-6t+21t^2+4t^3\bigr)c_1^6
			+6\bigl(1-3t+10t^2\bigr)c_1^4c_2
			+\bigl(-4+72t\bigr)c_2^3 \notag\\
			&\quad
			+12\bigl(1-3t+12t^2\bigr)c_1^3c_3
			+12\bigl(3-8t\bigr)c_1c_2c_3
			-60c_3^2
			+72c_2c_4 \notag\\
			&\quad
			-3c_1^2\Bigl(\bigl(7-8t+56t^2\bigr)c_2^2+12\bigl(1-2t\bigr)c_4\Bigr).
		\end{align}	
		Now using Lemma \ref{lemma3} in  \eqref{eq:3.15} we obtain 
		\begin{align}\label{eq:3.16}
			8640\, H_{3}(1)= A_{1}+B_{1} \eta +C_{1}\eta^2+ D_{1}\rho
		\end{align}
		where, 	
		\begin{align*}
			A_1={}&-\bigl(1-6t+21t^2+4t^3\bigr)c_1^6
			-6\bigl(1-3t+10t^2\bigr)\gamma c_1^4\bigl(-1+c_1^2\bigr)
			+12\gamma^4 c_1^2\bigl(-1+c_1^2\bigr)^2 \notag\\
			&\quad
			-3\gamma^2 c_1^2\bigl(-1+c_1^2\bigr)
			\Bigl(-7+3c_1^2+8t^2\bigl(-7+c_1^2\bigr)+4t\bigl(2+c_1^2\bigr)\Bigr) \notag\\
			&\quad
			-4\gamma^3\bigl(-1+c_1^2\bigr)
			\Bigl(-1-7c_1^2-c_1^4+6t\bigl(3-2c_1^2+2c_1^4\bigr)\Bigr),\notag\\[2mm]
			B_1={}&12\bigl(-1+|\gamma|^2\bigr)c_1\bigl(-1+c_1^2\bigr)
			\Bigl(\bigl(1-3t+12t^2\bigr)c_1^2
			+2\gamma^2\bigl(-1+c_1^2\bigr)\\
			&\quad+3\gamma\bigl(1+c_1^2\bigr)
			-4t\gamma\bigl(2+c_1^2\bigr)\Bigr),\notag\\[2mm]
			C_1={}&-12\bigl(-1+|\gamma|^2\bigr)\bigl(-1+c_1^2\bigr)
			\Bigl(5-5c_1^2+5|\gamma|^2\bigl(-1+c_1^2\bigr) \notag\\
			&\quad
			-3\Bigl(\bigl(1-2t\bigr)c_1^2+2\gamma\bigl(-1+c_1^2\bigr)\Bigr)\overline{\gamma}\Bigr),\notag\\[2mm]
			D_1={}&36\bigl(-1+|\gamma|^2\bigr)\bigl(-1+|\eta|^2\bigr)\bigl(-1+c_1^2\bigr)
			\Bigl(\bigl(1-2t\bigr)c_1^2+2\gamma\bigl(-1+c_1^2\bigr)\Bigr).
		\end{align*}
		Taking modulus on both sides of \eqref{eq:3.16}, and using
		$|\gamma|=x$, $|\eta|=y$, $|\rho|\le 1$, together with the triangle inequality, we obtain
		\[
		8640\,|H_3(1)|\le H(p_1,x,y,t),
		\]
		where
		\begin{align}
			H(p_1,x,y,t)
			&:=\bigl(1-6t+21t^2+4t^3\bigr)p_1^6 
			+x\Bigl(6\bigl(1-3t+10t^2\bigr)p_1^4(1-p_1^2)\Bigr) \notag\\
			&\quad	+x^2\Bigl(3p_1^2(1-p_1^2)\bigl(7-3p_1^2+8t^2(7-p_1^2)-4t(2+p_1^2)\bigr)\Bigr) \notag\\
			&\quad
			+x^3\Bigl(4(1-p_1^2)\bigl(1+7p_1^2+p_1^4+6t(3-2p_1^2+2p_1^4)\bigr)\Bigr) \notag\\
			&\quad
			+x^4\Bigl(12p_1^2(-1+p_1^2)^2\Bigr) \notag\\
			&\quad
			+y\big(12(1-x^2)p_1(1-p_1^2)\bigl((1-3t+12t^2)p_1^2+2x^2(1-p_1^2) \notag\\
			&\quad
			+3x(1+p_1^2)+4tx(2+p_1^2)\bigr)\big) \notag\\
			&\quad
			+y^2\Bigl(12(1-x^2)(1-p_1^2)\bigl(5(1-x^2)(1-p_1^2) \notag\\
			&\quad
			+3\bigl((1-2t)p_1^2+2x(1-p_1^2)\bigr)x\bigr)\Bigr) \notag\\
			&\quad
			+36(1-x^2)(1-y^2)(1-p_1^2)\bigl((1-2t)p_1^2+2x(1-p_1^2)\bigr).
		\end{align}
		It is easy to see that the coefficient of the term linear in  $y,$ precisely 
		$$\Bigl(12(1-x^2)p_1(1-p_1^2)\bigl((1-3t+12t^2)p_1^2+2x^2(1-p_1^2)
		+3x(1+p_1^2)+4tx(2+p_1^2)\bigr)\Bigr)$$
		remains non-negative throughout the region $\mathcal{D}:=\{
		(p_1, x, y, t)\in [0, 1]\times [0, 1]\times [0, 1]\times [0, 1/2]\}.$ Thus, to simplify the estimate, we replace this linear term by its value at $y=1$ for this term and we denote the new function by $H_{1}(p_1, x, y, t),$ which satisfies the property $H(p_1, x, y, t)\le H_{1}(p_1, x, y, t)$ for every $(p_1, x, y, t)\in \mathcal{D}.$ We write $H_{1}(p_1, x, y, t)$ as follows:
		\begin{align}
			H_1(p_1,x,y,t)
			&:=
			\left(1-6t+21t^2+4t^3\right)p_1^6 
			+6\left(1-3t+10t^2\right)p_1^4(1-p_1^2)\,x \notag\\
			&\quad
			+3p_1^2(1-p_1^2)\left(7-3p_1^2+8t^2(7-p_1^2)-4t(2+p_1^2)\right)x^2 \notag\\
			&\quad
			+4(1-p_1^2)\big(1+7p_1^2+p_1^4+6t(3-2p_1^2+2p_1^4)\big)x^3 \notag\\
			&\quad
			+12p_1^2(-1+p_1^2)^2x^4 
			+12(1-x^2)p_1(1-p_1^2)
			\big((1-3t+12t^2)p_1^2\notag\\
			&\quad+2x^2(1-p_1^2)+3x(1+p_1^2)\notag\\
			&\quad+4tx(2+p_1^2)\big) 
			+12y^2(1-x^2)(1-p_1^2)
			\big(5(1-x^2)(1-p_1^2)\notag\\
			&\quad+3\bigl((1-2t)p_1^2+2x(1-p_1^2)\bigr)x\big) \notag\\
			&\quad
			+36(1-x^2)(1-y^2)(1-p_1^2)\left((1-2t)p_1^2+2x(1-p_1^2)\right).
		\end{align}
		Since $H_1$ is a quadratic polynomial in $y$ with no linear term, its maximum occurs at either $y=0$ or $y=1$. We denote $G_{1}(p_1, x, t):=H_{1}(p_1, x, 1, t)$ and $G_{2}(p_1, x, t):=H_{1}(p_1, x, 0, t).$ First we deal with the case $y=1.$ 
		For simplicity, write $p=p_1.$
		Next, we want to determine
		\begin{align*}
			\max\left\{G_{1}(p,x,t):0\le p\le 1,\ 0\le x\le 1,\ 0\le t\le \frac12\right\}.
		\end{align*}
		We first eliminate the variable \(t\). A direct computation gives
		\[
		\frac{\partial^2 G_{1}}{\partial t^2}
		=
		6p^2\Bigl(
		4p^4t+8p^4x^2-20p^4x+7p^4+48p^3x^2-48p^3-64p^2x^2+20p^2x-48px^2+48p+56x^2
		\Bigr).
		\]
		Rewriting the bracket, 
		\[
		8(p-1)^2(p+1)(p+7)x^2+20p^2(1-p^2)x+p\bigl(4p^3t+7p^3-48p^2+48\bigr).
		\]
		Since
		\[
		8(p-1)^2(p+1)(p+7)x^2\ge 0,
		\qquad
		20p^2(1-p^2)x\ge 0,
		\]
		it remains to estimate the last term. For \(0\le t\le1/2\), 
		\begin{align*}
			4p^3t+7p^3-48p^2+48\ge 7p^3-48p^2+48.
		\end{align*}
		Now
		\[
		\frac{d}{dp}\bigl(7p^3-48p^2+48\bigr)=21p^2-96p<0,
		\qquad 0\le p\le 1,
		\]
		so the minimum on \([0,1]\) occurs at \(p=1\), where the value is \(7\). Hence
		\[
		\frac{\partial^2 G_{1}}{\partial t^2}\ge 0
		\qquad\text{on }[0,1]\times[0,1]\times\left[0,\frac12\right].
		\]
		Therefore \(G_{1}(p, x, t)\) is convex in \(t\), and so
		\begin{align*}
			\max_{0\le t\le 1/2}G_{1}(p,x,t)=\max\left\{G_{1}(p,x,0),\,G_{1}\left(p,x,\frac12\right)\right\}.
		\end{align*}
		Setting, 
		\[
		G_0(p,x):=G_{1}(p,x,0),
		\qquad
		G_{1/2}(p,x):=G_{1}\left(p,x,\frac12\right).
		\]
		We write $G_{0}(p, x)$ and $G_{1/2}(p, x)$ in their extended form such that 
		\[
		\begin{aligned}
			G_0(p,x)=\;&
			12p^6x^4-4p^6x^3+9p^6x^2-6p^6x+p^6
			-24p^5x^4+36p^5x^3\\
			&+36p^5x^2-36p^5x-12p^5-36p^4x^4+12p^4x^3-78p^4x^2\\
			&-30p^4x+60p^4
			+48p^3x^4-60p^3x^2+12p^3+36p^2x^4-12p^2x^3\\
			&+117p^2x^2+36p^2x-120p^2
			-24px^4-36px^3+24px^2+36px\\
			&-12x^4+4x^3-48x^2+60,
		\end{aligned}
		\]
		and
		\[
		\begin{aligned}
			G_{1/2}(p,x)=\;&
			12p^6x^4-28p^6x^3+21p^6x^2-12p^6x+\frac{15}{4}p^6
			-24p^5x^4+60p^5x^3+54p^5x^2\\
			&-60p^5x-30p^5-36p^4x^4+24p^4x^3-120p^4x^2+12p^4x+60p^4
			+48p^3x^4\\
			&+24p^3x^3-78p^3x^2-24p^3x+30p^3+36p^2x^4-36p^2x^3+147p^2x^2\\
			&-120p^2
			-24px^4-84px^3+24px^2+84px-12x^4+40x^3-48x^2+60.
		\end{aligned}
		\]
		We now express each polynomial in the  Bernstein basis
		\begin{align*}
			\beta_i^{(6)}(p)\beta_j^{(4)}(x),
			\qquad
			\beta_i^{(m)}(s)=\binom{m}{i}s^i(1-s)^{m-i}.
		\end{align*}
		Thus, we have 
		\begin{align*}
			G_0(p,x)&=\sum_{i=0}^6\sum_{j=0}^4 b^{(0)}_{ij}\,\beta_i^{(6)}(p)\beta_j^{(4)}(x),
			\\[2mm]
			G_{1/2}(p,x)&=\sum_{i=0}^6\sum_{j=0}^4 b^{(1/2)}_{ij}\,\beta_i^{(6)}(p)\beta_j^{(4)}(x).
		\end{align*}
		The Bernstein coefficients for $G_{0}$ are
		\[
		B^{(0)}=(b^{(0)}_{ij})=
		\begin{pmatrix}
			60 & 60 & 52 & 37 & 4\\[5mm]
			60 & \frac{123}{2} & \frac{167}{3} & 42 & 4\\[5mm]
			52 & \frac{278}{5} & \frac{323}{6} & \frac{89}{2} & \frac{39}{5}\\[5mm]
			\frac{183}{5} & \frac{429}{10} & \frac{233}{5} & \frac{218}{5} & \frac{77}{5}\\[5mm]
			\frac{92}{5} & \frac{55}{2} & \frac{181}{5} & \frac{77}{2} & 22\\[5mm]
			4 & \frac{27}{2} & 23 & \frac{53}{2} & 18\\[5mm]
			1 & 1 & 1 & 1 & 1
		\end{pmatrix}.
		\]\\
		Therefore, we have 
		\begin{align*}
			G_0(p,x)\le \max B^{(0)}=\frac{123}{2}.
		\end{align*}
		The Bernstein coefficients for $G_{1/2}$ are
		\[
		B^{(1/2)}=(b^{(1/2)}_{ij})=
		\begin{pmatrix}
			60 & 60 & 52 & 46 & 40\\[5mm]
			60 & \frac{127}{2} & \frac{179}{3} & 55 & 40\\[5mm]
			52 & 59 & \frac{1829}{30} & \frac{603}{10} & \frac{209}{5}\\[5mm]
			\frac{75}{2} & \frac{477}{10} & \frac{1123}{20} & \frac{1217}{20} & \frac{227}{5}\\[5mm]
			22 & 35 & \frac{728}{15} & \frac{283}{5} & \frac{234}{5}\\[5mm]
			10 & 23 & 36 & 43 & 38\\[5mm]
			\frac{15}{4} & \frac{15}{4} & \frac{15}{4} & \frac{15}{4} & \frac{15}{4}
		\end{pmatrix}.
		\]\\
		Thus, we have
		\begin{align*}
			G_{1/2}(p,x)\le \max B^{(1/2)}=\frac{127}{2}.
		\end{align*}
		These first bounds are not sharp enough, so we subdivide the square \([0,1]\times[0, 1]\) into
		\begin{align*}
			Q_1&=\left[0,\frac12\right]\times\left[0,\frac12\right],\qquad
			Q_2=\left[0,\frac12\right]\times\left[\frac12,1\right],
			\\[4mm]
			Q_3&=\left[\frac12,1\right]\times\left[0,\frac12\right],\qquad
			Q_4=\left[\frac12,1\right]\times\left[\frac12,1\right].
		\end{align*}
		To compute the Bernstein coefficients on each sub-regions, we map each \(Q_r\) affinely onto the unit square \([0,1]\times[0, 1]\). If \((u,v)\in[0,1]\times[0, 1]\), then the affine changes are
		\[
		\begin{aligned}
			Q_1:\quad & p=\frac{u}{2},\qquad x=\frac{v}{2},\\[1mm]
			Q_2:\quad & p=\frac{u}{2},\qquad x=\frac{1+v}{2},\\[1mm]
			Q_3:\quad & p=\frac{1+u}{2},\qquad x=\frac{v}{2},\\[1mm]
			Q_4:\quad & p=\frac{1+u}{2},\qquad x=\frac{1+v}{2}.
		\end{aligned}
		\]
		We use the  Bernstein basis
		\[
		\beta_i^{(6)}(u)=\binom{6}{i}u^i(1-u)^{6-i},
		\qquad
		\beta_j^{(4)}(v)=\binom{4}{j}v^j(1-v)^{4-j}.
		\]
		Thus, for each \(r=1,2,3,4\), the transformed polynomial can be written in the form
		\[
		G_0\bigl(p(u),x(v)\bigr)
		=
		\sum_{i=0}^{6}\sum_{j=0}^{4} b^{(Q_r)}_{ij}\,\beta_i^{(6)}(u)\beta_j^{(4)}(v).
		\]
		For \(Q_1\), the Bernstein coefficient matrix is
		\[
		B_{Q_1}^{(0)}=
		\begin{pmatrix}
			60 & 60 & 58 & \frac{433}{8} & \frac{191}{4}\\[3mm]
			60 & \frac{483}{8} & \frac{353}{6} & \frac{1773}{32} & \frac{197}{4}\\[3mm]
			58 & \frac{2353}{40} & \frac{27791}{480} & \frac{1103}{20} & \frac{991}{20}\\[3mm]
			\frac{2163}{40} & \frac{2217}{40} & \frac{17681}{320} & \frac{17083}{320} & \frac{973}{20}\\[3mm]
			\frac{971}{20} & \frac{3231}{64} & \frac{49117}{960} & \frac{32193}{640} & \frac{14929}{320}\\[3mm]
			\frac{671}{16} & \frac{5691}{128} & \frac{17663}{384} & \frac{23631}{512} & \frac{2795}{64}\\[3mm]
			\frac{2233}{64} & \frac{9697}{256} & \frac{20543}{512} & \frac{5267}{128} & \frac{5087}{128}
		\end{pmatrix}.
		\]\\
		For \(Q_2\), the Bernstein coefficient matrix is
		\[
		B_{Q_2}^{(0)}=
		\begin{pmatrix}
			\frac{191}{4} & \frac{331}{8} & \frac{65}{2} & \frac{41}{2} & 4\\[3mm]
			\frac{197}{4} & \frac{1379}{32} & \frac{821}{24} & \frac{87}{4} & 4\\[3mm]
			\frac{991}{20} & \frac{879}{20} & \frac{17039}{480} & \frac{1853}{80} & \frac{99}{20}\\[3mm]
			\frac{973}{20} & \frac{14053}{320} & \frac{11621}{320} & \frac{3949}{160} & \frac{137}{20}\\[3mm]
			\frac{14929}{320} & \frac{27523}{640} & \frac{35107}{960} & \frac{8353}{320} & \frac{47}{5}\\[3mm]
			\frac{2795}{64} & \frac{21089}{512} & \frac{6925}{192} & \frac{1731}{64} & 12\\[3mm]
			\frac{5087}{128} & \frac{4907}{128} & \frac{17663}{512} & \frac{1729}{64} & \frac{223}{16}
		\end{pmatrix}.
		\]\\
		For \(Q_3\), the Bernstein coefficient matrix is
		\[
		B_{Q_3}^{(0)}=
		\begin{pmatrix}
			\frac{2233}{64} & \frac{9697}{256} & \frac{20543}{512} & \frac{5267}{128} & \frac{5087}{128}\\[3mm]
			\frac{891}{32} & \frac{2003}{64} & \frac{26303}{768} & \frac{18505}{512} & \frac{573}{16}\\[3mm]
			\frac{1629}{80} & \frac{773}{32} & \frac{53119}{1920} & \frac{9689}{320} & \frac{9899}{320}\\[3mm]
			\frac{131}{10} & \frac{2717}{160} & \frac{3319}{160} & \frac{15207}{640} & \frac{4023}{160}\\[3mm]
			\frac{137}{20} & \frac{829}{80} & \frac{277}{20} & \frac{67}{4} & \frac{92}{5}\\[3mm]
			\frac{5}{2} & \frac{39}{8} & \frac{29}{4} & \frac{37}{4} & \frac{21}{2}\\[3mm]
			1 & 1 & 1 & 1 & 1
		\end{pmatrix}.
		\]\\
		For \(Q_4\), the Bernstein coefficient matrix is
		\[
		B_{Q_4}^{(0)}=
		\begin{pmatrix}
			\frac{5087}{128} & \frac{4907}{128} & \frac{17663}{512} & \frac{1729}{64} & \frac{223}{16}\\[3mm]
			\frac{573}{16} & \frac{18167}{512} & \frac{25289}{768} & \frac{1727}{64} & \frac{127}{8}\\[3mm]
			\frac{9899}{320} & \frac{10109}{320} & \frac{58159}{1920} & \frac{8313}{320} & \frac{343}{20}\\[3mm]
			\frac{4023}{160} & \frac{16977}{640} & \frac{1051}{40} & \frac{47}{2} & \frac{341}{20}\\[3mm]
			\frac{92}{5} & \frac{401}{20} & \frac{409}{20} & \frac{303}{16} & \frac{59}{4}\\[3mm]
			\frac{21}{2} & \frac{47}{4} & \frac{49}{4} & \frac{93}{8} & \frac{19}{2}\\[3mm]
			1 & 1 & 1 & 1 & 1
		\end{pmatrix}.
		\]\\
		Hence, by the Bernstein enclosure principle on each sub-regions,
		\[
		\max_{Q_1}G_0\le \frac{483}{8},\qquad
		\max_{Q_2}G_0\le \frac{991}{20},\qquad
		\max_{Q_3}G_0\le \frac{5267}{128},\qquad
		\max_{Q_4}G_0\le \frac{5087}{128}.
		\]
		It is easy to see that all three bounds are less than $60$, except on the region $Q_1$.\\[2mm]
		For $G_{1/2}$,
		The Bernstein coefficient matrices for \(G_{1/2}\) on the four sub-regions are\\
		\[
		B_{Q_1}^{(1/2)}=
		\begin{pmatrix}
			60 & 60 & 58 & \frac{221}{4} & \frac{209}{4}\\[3mm]
			60 & \frac{487}{8} & \frac{359}{6} & \frac{1853}{32} & \frac{221}{4}\\[3mm]
			58 & \frac{239}{4} & \frac{9563}{160} & \frac{1177}{20} & \frac{2273}{40}\\[3mm]
			\frac{867}{16} & \frac{9087}{160} & \frac{37079}{640} & \frac{1861}{32} & \frac{18241}{320}\\[3mm]
			49 & \frac{8389}{160} & \frac{8753}{160} & \frac{3591}{64} & \frac{17909}{320}\\[3mm]
			\frac{1375}{32} & \frac{6035}{128} & \frac{38761}{768} & \frac{27123}{512} & \frac{6891}{128}\\[3mm]
			\frac{9375}{256} & \frac{10581}{256} & \frac{23265}{512} & \frac{12489}{256} & \frac{12921}{256}
		\end{pmatrix},
		\]\\[2mm]
		\[
		B_{Q_2}^{(1/2)}=
		\begin{pmatrix}
			\frac{209}{4} & \frac{197}{4} & 46 & 43 & 40\\[3mm]
			\frac{221}{4} & \frac{1683}{32} & \frac{1181}{24} & \frac{181}{4} & 40\\[3mm]
			\frac{2273}{40} & \frac{274}{5} & \frac{8267}{160} & \frac{3781}{80} & \frac{809}{20}\\[3mm]
			\frac{18241}{320} & \frac{1117}{20} & \frac{34127}{640} & \frac{15671}{320} & \frac{827}{20}\\[3mm]
			\frac{17909}{320} & \frac{17863}{320} & \frac{8661}{160} & \frac{1607}{32} & \frac{849}{20}\\[3mm]
			\frac{6891}{128} & \frac{28005}{512} & \frac{41407}{768} & \frac{6485}{128} & \frac{173}{4}\\[3mm]
			\frac{12921}{256} & \frac{13353}{256} & \frac{26721}{512} & \frac{6357}{128} & \frac{10995}{256}
		\end{pmatrix},
		\]\\
		\[
		B_{Q_3}^{(1/2)}=
		\begin{pmatrix}
			\frac{9375}{256} & \frac{10581}{256} & \frac{23265}{512} & \frac{12489}{256} & \frac{12921}{256}\\[3mm]
			\frac{3875}{128} & \frac{2273}{64} & \frac{15517}{384} & \frac{22833}{512} & \frac{3015}{64}\\[3mm]
			\frac{1511}{64} & \frac{9333}{320} & \frac{66401}{1920} & \frac{5037}{128} & \frac{3401}{80}\\[3mm]
			\frac{549}{32} & \frac{3627}{160} & \frac{18033}{640} & \frac{4235}{128} & \frac{11693}{320}\\[3mm]
			\frac{183}{16} & \frac{261}{16} & \frac{5093}{240} & \frac{128}{5} & \frac{2303}{80}\\[3mm]
			\frac{55}{8} & \frac{81}{8} & \frac{107}{8} & \frac{65}{4} & \frac{147}{8}\\[3mm]
			\frac{15}{4} & \frac{15}{4} & \frac{15}{4} & \frac{15}{4} & \frac{15}{4}
		\end{pmatrix},
		\]\\
		and
		\[
		B_{Q_4}^{(1/2)}=
		\begin{pmatrix}
			\frac{12921}{256} & \frac{13353}{256} & \frac{26721}{512} & \frac{6357}{128} & \frac{10995}{256}\\[3mm]
			\frac{3015}{64} & \frac{25407}{512} & \frac{9689}{192} & \frac{6229}{128} & \frac{5459}{128}\\[3mm]
			\frac{3401}{80} & \frac{29231}{640} & \frac{90677}{1920} & \frac{1479}{32} & \frac{13199}{320}\\[3mm]
			\frac{11693}{320} & \frac{25597}{640} & \frac{26877}{640} & \frac{13339}{320} & \frac{6071}{160}\\[3mm]
			\frac{2303}{80} & \frac{1279}{40} & \frac{8153}{240} & \frac{2729}{80} & \frac{2531}{80}\\[3mm]
			\frac{147}{8} & \frac{41}{2} & \frac{175}{8} & \frac{177}{8} & \frac{167}{8}\\[3mm]
			\frac{15}{4} & \frac{15}{4} & \frac{15}{4} & \frac{15}{4} & \frac{15}{4}
		\end{pmatrix}.
		\]\\
		Therefore, the correct Bernstein upper bounds are
		\[
		\max_{Q_1} G_{1/2}\le \frac{487}{8},\qquad
		\max_{Q_2} G_{1/2}\le \frac{18241}{320},
		\]
		\[
		\max_{Q_3} G_{1/2}\le \frac{12921}{256},\qquad
		\max_{Q_4} G_{1/2}\le \frac{26721}{512}.
		\]
		In particular, \(Q_2,Q_3,Q_4\) already give bounds strictly less than \(60\), whereas \(Q_1\) still needs further subdivision.
		We now subdivide the remaining region
		\[
		Q_1=\left[0,\frac12\right]\times\left[0,\frac12\right]
		\]
		into the four sub-regions
		\begin{align*}
			Q_{11}&=\left[0,\frac14\right]\times\left[0,\frac14\right],\qquad
			Q_{12}=\left[0,\frac14\right]\times\left[\frac14,\frac12\right],\\[2mm]
			Q_{13}&=\left[\frac14,\frac12\right]\times\left[0,\frac14\right],\qquad
			Q_{14}=\left[\frac14,\frac12\right]\times\left[\frac14,\frac12\right].
		\end{align*}
		To compute the Bernstein coefficients on each sub-regions, we again map each \(Q_{1j}\) affinely onto the unit square \([0, 1]\times[0, 1]\). If \((u,v)\in[0, 1]\times[0, 1]\), then the affine changes are
		\[
		\begin{aligned}
			Q_{11}:\quad & p=\frac{u}{4},\qquad x=\frac{v}{4},\\[1mm]
			Q_{12}:\quad & p=\frac{u}{4},\qquad x=\frac{1+v}{4},\\[1mm]
			Q_{13}:\quad & p=\frac{1+u}{4},\qquad x=\frac{v}{4},\\[1mm]
			Q_{14}:\quad & p=\frac{1+u}{4},\qquad x=\frac{1+v}{4}.
		\end{aligned}
		\]
		We use the  Bernstein basis
		\[
		\beta_i^{(6)}(u)=\binom{6}{i}u^i(1-u)^{6-i},
		\qquad
		\beta_j^{(4)}(v)=\binom{4}{j}v^j(1-v)^{4-j}.
		\]
		Thus, on each sub-regions,
		\[
		G_0\bigl(p(u),x(v)\bigr)
		=
		\sum_{i=0}^{6}\sum_{j=0}^{4} b_{ij}\,\beta_i^{(6)}(u)\beta_j^{(4)}(v).
		\]
		For \(Q_{11}\), the Bernstein coefficient matrix is
		\[
		B_{Q_{11}}^{(0)}=
		\begin{pmatrix}
			60 & 60 & \frac{119}{2} & \frac{3745}{64} & \frac{3649}{64}\\[3mm]
			60 & \frac{1923}{32} & \frac{5731}{96} & \frac{30117}{512} & \frac{14701}{256}\\[3mm]
			\frac{119}{2} & \frac{19103}{320} & \frac{456343}{7680} & \frac{300401}{5120} & \frac{293907}{5120}\\[3mm]
			\frac{18723}{320} & \frac{9411}{160} & \frac{600823}{10240} & \frac{594667}{10240} & \frac{1166459}{20480}\\[3mm]
			\frac{18257}{320} & \frac{117727}{2048} & \frac{235457}{4096} & \frac{4673237}{81920} & \frac{4595849}{81920}\\[3mm]
			\frac{28247}{512} & \frac{456537}{8192} & \frac{915641}{16384} & \frac{7290803}{131072} & \frac{3595905}{65536}\\[3mm]
			\frac{216721}{4096} & \frac{1756493}{32768} & \frac{7067867}{131072} & \frac{14115709}{262144} & \frac{13970539}{262144}
		\end{pmatrix}.
		\]\\
		For \(Q_{12}\), the Bernstein coefficient matrix is\\
		\[
		B_{Q_{12}}^{(0)}=
		\begin{pmatrix}
			\frac{3649}{64} & \frac{3553}{64} & \frac{107}{2} & \frac{815}{16} & \frac{191}{4}\\[3mm]
			\frac{14701}{256} & \frac{28687}{512} & \frac{20779}{384} & \frac{6609}{128} & \frac{97}{2}\\[3mm]
			\frac{293907}{5120} & \frac{287413}{5120} & \frac{417379}{7680} & \frac{33271}{640} & \frac{979}{20}\\[3mm]
			\frac{1166459}{20480} & \frac{35737}{640} & \frac{555073}{10240} & \frac{266233}{5120} & \frac{491}{10}\\[3mm]
			\frac{4595849}{81920} & \frac{4518461}{81920} & \frac{1099897}{20480} & \frac{1058427}{20480} & \frac{250657}{5120}\\[3mm]
			\frac{3595905}{65536} & \frac{7092817}{131072} & \frac{1732289}{32768} & \frac{1672755}{32768} & \frac{24847}{512}\\[3mm]
			\frac{13970539}{262144} & \frac{13825369}{262144} & \frac{6777527}{131072} & \frac{1642283}{32768} & \frac{391817}{8192}
		\end{pmatrix}.
		\]\\
		For \(Q_{13}\), the Bernstein coefficient matrix is\\
		\[
		B_{Q_{13}}^{(0)}=
		\begin{pmatrix}
			\frac{216721}{4096} & \frac{1756493}{32768} & \frac{7067867}{131072} & \frac{14115709}{262144} & \frac{13970539}{262144}\\[3mm]
			\frac{103733}{2048} & \frac{843419}{16384} & \frac{3405303}{65536} & \frac{3412453}{65536} & \frac{6778729}{131072}\\[3mm]
			\frac{245837}{5120} & \frac{401253}{8192} & \frac{1626395}{32768} & \frac{16363463}{327680} & \frac{16317991}{327680}\\[3mm]
			\frac{115337}{2560} & \frac{945367}{20480} & \frac{3849131}{81920} & \frac{121573}{2560} & \frac{7794061}{163840}\\[3mm]
			\frac{53541}{1280} & \frac{441089}{10240} & \frac{1083137}{24576} & \frac{3667901}{81920} & \frac{738539}{16384}\\[3mm]
			\frac{4917}{128} & \frac{40747}{1024} & \frac{503237}{12288} & \frac{171395}{4096} & \frac{347017}{8192}\\[3mm]
			\frac{2233}{64} & \frac{18629}{512} & \frac{77195}{2048} & \frac{158743}{4096} & \frac{161659}{4096}
		\end{pmatrix}.
		\]\\
		For \(Q_{14}\), the Bernstein coefficient matrix is\\
		\[
		B_{Q_{14}}^{(0)}=
		\begin{pmatrix}
			\frac{13970539}{262144} & \frac{13825369}{262144} & \frac{6777527}{131072} & \frac{1642283}{32768} & \frac{391817}{8192}\\[3mm]
			\frac{6778729}{131072} & \frac{841569}{16384} & \frac{3312949}{65536} & \frac{1611811}{32768} & \frac{193041}{4096}\\[3mm]
			\frac{16317991}{327680} & \frac{16272519}{327680} & \frac{8041031}{163840} & \frac{982247}{20480} & \frac{472639}{10240}\\[3mm]
			\frac{7794061}{163840} & \frac{780745}{16384} & \frac{3875909}{81920} & \frac{380553}{8192} & \frac{45999}{1024}\\[3mm]
			\frac{738539}{16384} & \frac{3717489}{81920} & \frac{5564449}{122880} & \frac{457597}{10240} & \frac{111193}{2560}\\[3mm]
			\frac{347017}{8192} & \frac{87811}{2048} & \frac{528599}{12288} & \frac{87407}{2048} & \frac{10677}{256}\\[3mm]
			\frac{161659}{4096} & \frac{164575}{4096} & \frac{83027}{2048} & \frac{5177}{128} & \frac{5087}{128}
		\end{pmatrix}.
		\]\\
		Hence, by the Bernstein enclosure principle on each sub-regions,
		\[
		\max_{Q_{11}}G_0\le \frac{1923}{32},\qquad
		\max_{Q_{12}}G_0\le \frac{14701}{256},
		\]
		\[
		\max_{Q_{13}}G_0\le \frac{7067867}{131072},\qquad
		\max_{Q_{14}}G_0\le \frac{13970539}{262144}.
		\]
		In particular,
		\[
		\frac{14701}{256}<60,\qquad
		\frac{7067867}{131072}<60,\qquad
		\frac{13970539}{262144}<60,
		\]
		so among these four sub-regions only \(Q_{11}\) still needs further refinement.\\
		For $G_{1/2}$,
		We now subdivide the remaining region
		\[
		Q_1=\left[0,\frac12\right]\times\left[0,\frac12\right]
		\]
		into the four sub-regions
		\begin{align*}
			Q_{11}&=\left[0,\frac14\right]\times\left[0,\frac14\right],\qquad
			Q_{12}=\left[0,\frac14\right]\times\left[\frac14,\frac12\right],\\[2mm]
			Q_{13}&=\left[\frac14,\frac12\right]\times\left[0,\frac14\right],\qquad
			Q_{14}=\left[\frac14,\frac12\right]\times\left[\frac14,\frac12\right].
		\end{align*}
		To compute the Bernstein coefficients on each sub-regions, we map each \(Q_{1j}\) affinely onto the unit square \([0, 1]\times[0, 1]\). If \((u,v)\in[0, 1]\times[0, 1]\), then the affine changes are
		\[
		\begin{aligned}
			Q_{11}:\quad & p=\frac{u}{4},\qquad x=\frac{v}{4},\\[1mm]
			Q_{12}:\quad & p=\frac{u}{4},\qquad x=\frac{1+v}{4},\\[1mm]
			Q_{13}:\quad & p=\frac{1+u}{4},\qquad x=\frac{v}{4},\\[1mm]
			Q_{14}:\quad & p=\frac{1+u}{4},\qquad x=\frac{1+v}{4}.
		\end{aligned}
		\]
		
		We use the  Bernstein basis
		\[
		\beta_i^{(6)}(u)=\binom{6}{i}u^i(1-u)^{6-i},
		\qquad
		\beta_j^{(4)}(v)=\binom{4}{j}v^j(1-v)^{4-j}.
		\]
		Thus, on each sub-regions,
		\[
		G_{1/2}\bigl(p(u),x(v)\bigr)
		=
		\sum_{i=0}^{6}\sum_{j=0}^{4} b_{ij}\,\beta_i^{(6)}(u)\beta_j^{(4)}(v).
		\]
		For \(Q_{11}\), the Bernstein coefficient matrix is
		\[
		B_{Q_{11}}^{(1/2)}=
		\begin{pmatrix}
			60 & 60 & \frac{119}{2} & \frac{1877}{32} & \frac{3685}{64}\\[3mm]
			60 & \frac{1927}{32} & \frac{5755}{96} & \frac{30377}{512} & \frac{14965}{256}\\[3mm]
			\frac{119}{2} & \frac{959}{16} & \frac{460049}{7680} & \frac{60951}{1024} & \frac{301427}{5120}\\[3mm]
			\frac{7491}{128} & \frac{151497}{2560} & \frac{1216171}{20480} & \frac{2426941}{40960} & \frac{1205039}{20480}\\[3mm]
			\frac{3655}{64} & \frac{296857}{5120} & \frac{448723}{7680} & \frac{2398229}{40960} & \frac{4783221}{81920}\\[3mm]
			\frac{56635}{1024} & \frac{461947}{8192} & \frac{1870411}{32768} & \frac{7531625}{131072} & \frac{3772097}{65536}\\[3mm]
			\frac{871215}{16384} & \frac{27885}{512} & \frac{7260375}{131072} & \frac{14688015}{262144} & \frac{14781075}{262144}
		\end{pmatrix}.
		\]\\
		For \(Q_{12}\), the Bernstein coefficient matrix is
		\[
		B_{Q_{12}}^{(1/2)}=
		\begin{pmatrix}
			\frac{3685}{64} & \frac{113}{2} & \frac{883}{16} & \frac{215}{4} & \frac{209}{4}\\[3mm]
			\frac{14965}{256} & \frac{29483}{512} & \frac{21679}{384} & \frac{7061}{128} & \frac{215}{4}\\[3mm]
			\frac{301427}{5120} & \frac{298099}{5120} & \frac{440081}{7680} & \frac{35959}{640} & \frac{8783}{160}\\[3mm]
			\frac{1205039}{20480} & \frac{478643}{8192} & \frac{236489}{4096} & \frac{290929}{5120} & \frac{142553}{2560}\\[3mm]
			\frac{4783221}{81920} & \frac{74531}{1280} & \frac{3550073}{61440} & \frac{146151}{2560} & \frac{287417}{5120}\\[3mm]
			\frac{3772097}{65536} & \frac{7556763}{131072} & \frac{470745}{8192} & \frac{1868191}{32768} & \frac{230457}{4096}\\[3mm]
			\frac{14781075}{262144} & \frac{14874135}{262144} & \frac{7446495}{131072} & \frac{1854975}{32768} & \frac{918825}{16384}
		\end{pmatrix}.
		\]\\
		For \(Q_{13}\), the Bernstein coefficient matrix is
		\[
		B_{Q_{13}}^{(1/2)}=
		\begin{pmatrix}
			\frac{871215}{16384} & \frac{27885}{512} & \frac{7260375}{131072} & \frac{14688015}{262144} & \frac{14781075}{262144}\\[3mm]
			\frac{418135}{8192} & \frac{430373}{8192} & \frac{3519553}{65536} & \frac{3578195}{65536} & \frac{7236881}{131072}\\[3mm]
			\frac{198975}{4096} & \frac{514779}{10240} & \frac{25399237}{491520} & \frac{17309657}{327680} & \frac{17596319}{327680}\\[3mm]
			\frac{93879}{2048} & \frac{488691}{10240} & \frac{4043133}{81920} & \frac{1039429}{20480} & \frac{1700161}{32768}\\[3mm]
			\frac{43919}{1024} & \frac{57539}{1280} & \frac{5750513}{122880} & \frac{3967591}{81920} & \frac{4080383}{81920}\\[3mm]
			\frac{20375}{512} & \frac{21513}{512} & \frac{541379}{12288} & \frac{188037}{4096} & \frac{389201}{8192}\\[3mm]
			\frac{9375}{256} & \frac{4989}{128} & \frac{84339}{2048} & \frac{177009}{4096} & \frac{184371}{4096}
		\end{pmatrix}.
		\]\\
		For \(Q_{14}\), the Bernstein coefficient matrix is
		\[
		B_{Q_{14}}^{(1/2)}=
		\begin{pmatrix}
			\frac{14781075}{262144} & \frac{14874135}{262144} & \frac{7446495}{131072} & \frac{1854975}{32768} & \frac{918825}{16384}\\[3mm]
			\frac{7236881}{131072} & \frac{1829343}{32768} & \frac{3680535}{65536} & \frac{1841759}{32768} & \frac{457911}{8192}\\[3mm]
			\frac{17596319}{327680} & \frac{17882981}{327680} & \frac{27119209}{491520} & \frac{142021}{2560} & \frac{1134653}{20480}\\[3mm]
			\frac{1700161}{32768} & \frac{4343089}{81920} & \frac{4413879}{81920} & \frac{2229949}{40960} & \frac{559207}{10240}\\[3mm]
			\frac{4080383}{81920} & \frac{838635}{16384} & \frac{1285453}{24576} & \frac{543941}{10240} & \frac{274061}{5120}\\[3mm]
			\frac{389201}{8192} & \frac{50291}{1024} & \frac{620141}{12288} & \frac{105507}{2048} & \frac{26703}{512}\\[3mm]
			\frac{184371}{4096} & \frac{191733}{4096} & \frac{99063}{2048} & \frac{12705}{256} & \frac{12921}{256}
		\end{pmatrix}.
		\]\\
		Hence, by the Bernstein enclosure principle on each sub-regions,
		\[
		\max_{Q_{11}}G_{1/2}\le \frac{1927}{32},\qquad
		\max_{Q_{12}}G_{1/2}\le \frac{301427}{5120},
		\]
		\[
		\max_{Q_{13}}G_{1/2}\le \frac{14781075}{262144},\qquad
		\max_{Q_{14}}G_{1/2}\le \frac{7446495}{131072}.
		\]
		
		In particular,
		\[
		\frac{301427}{5120}<60,\qquad
		\frac{14781075}{262144}<60,\qquad
		\frac{7446495}{131072}<60,
		\]
		so among these four sub-regions only \(Q_{11}\) still needs further refinement.
		Thus the only remaining region is
		\begin{align*}
			Q_{11}=\left[0,\frac14\right]\times\left[0,\frac14\right].
		\end{align*}
		The Bernstein coefficients of $G_{0}$ on $Q_{11}$ are
		\[
		B_{11}^{(0)}=
		\begin{pmatrix}
			60 & 60 & \frac{119}{2} & \frac{3745}{64} & \frac{3649}{64}\\[5mm]
			60 & \frac{1923}{32} & \frac{5731}{96} & \frac{30117}{512} & \frac{14701}{256}\\[5mm]
			\frac{119}{2} & \frac{19103}{320} & \frac{456343}{7680} & \frac{300401}{5120} & \frac{293907}{5120}\\[5mm]
			\frac{18723}{320} & \frac{9411}{160} & \frac{600823}{10240} & \frac{594667}{10240} & \frac{1166459}{20480}\\[5mm]
			\frac{18257}{320} & \frac{117727}{2048} & \frac{235457}{4096} & \frac{4673237}{81920} & \frac{4595849}{81920}\\[5mm]
			\frac{28247}{512} & \frac{456537}{8192} & \frac{915641}{16384} & \frac{7290803}{131072} & \frac{3595905}{65536}\\[5mm]
			\frac{216721}{4096} & \frac{1756493}{32768} & \frac{7067867}{131072} & \frac{14115709}{262144} & \frac{13970539}{262144}
		\end{pmatrix},
		\]\\
		whose maximum is
		\[
		\frac{1923}{32}=60+\frac{3}{32}.
		\]
		The Bernstein coefficients of $G_{1/2}$ on $Q_{11}$ are\\
		\[
		B_{11}^{(1/2)}=
		\begin{pmatrix}
			60 & 60 & \frac{119}{2} & \frac{1877}{32} & \frac{3685}{64}\\[5mm]
			60 & \frac{1927}{32} & \frac{5755}{96} & \frac{30377}{512} & \frac{14965}{256}\\[5mm]
			\frac{119}{2} & \frac{959}{16} & \frac{460049}{7680} & \frac{60951}{1024} & \frac{301427}{5120}\\[5mm]
			\frac{7491}{128} & \frac{151497}{2560} & \frac{1216171}{20480} & \frac{2426941}{40960} & \frac{1205039}{20480}\\[5mm]
			\frac{3655}{64} & \frac{296857}{5120} & \frac{448723}{7680} & \frac{2398229}{40960} & \frac{4783221}{81920}\\[5mm]
			\frac{56635}{1024} & \frac{461947}{8192} & \frac{1870411}{32768} & \frac{7531625}{131072} & \frac{3772097}{65536}\\[5mm]
			\frac{871215}{16384} & \frac{27885}{512} & \frac{7260375}{131072} & \frac{14688015}{262144} & \frac{14781075}{262144}
		\end{pmatrix},
		\]\\
		whose maximum is
		$
		{1927}/{32}=60+{7}/{32}.
		$
		This is still greater than $60$ on $Q_{11}$
		, which we solve directly.\\[2mm]
		On $Q_{11}$ we have
		$
		0\le p\le 1/4,\,  0\le x\le 1/4.
		$
		From the expansion of $G_{0}$, if we discard all negative terms, then
		\[
		\begin{aligned}
			G_0(p,x)&\le
			60-120p^2-48x^2+36px+24px^2+36p^2x+117p^2x^2+4x^3\\
			&\quad+60p^4+12p^3+36p^5x^2+36p^5x^3+12p^4x^3+48p^3x^4+36p^2x^4\\
			&\quad+p^6+9p^6x^2+12p^6x^4.
		\end{aligned}
		\]
		\begin{align*}
			24px^2 &\le 6px, &
			36p^2x &\le 9p^2, &
			117p^2x^2 &\le \frac{117}{16}p^2, &
			4x^3 &\le x^2,\\
			60p^4 &\le \frac{15}{4}p^2, &
			12p^3 &\le 3p^2, &
			36p^5x^2 &\le \frac{9}{16}p^2, &
			36p^5x^3 &\le \frac{9}{64}p^2,\\
			12p^4x^3 &\le \frac{3}{256}p^2, &
			48p^3x^4 &\le \frac{3}{64}p^2, &
			36p^2x^4 &\le \frac{9}{64}p^2,\\
			p^6 &\le \frac{1}{256}p^2, &
			9p^6x^2 &\le \frac{9}{256}p^2, &
			12p^6x^4 &\le \frac{3}{1024}p^2.
		\end{align*}
		Therefore
		\[
		G_0(p,x)\le 60-96p^2-47x^2+42px.
		\]
		Using
		\[
		42px\le 21p^2+21x^2,
		\]
		we get
		\[
		G_0(p,x)\le 60-75p^2-26x^2\le 60
		\qquad\text{on }Q_{11}.
		\]
		Now we handle $G_{1/2}$ in the same way. From its expansion,
		\[
		\begin{aligned}
			G_{1/2}(p,x)&\le
			60-120p^2-48x^2+84px+24px^2+147p^2x^2+40x^3+12p^4x\\
			&\quad+60p^4+30p^3+54p^5x^2+60p^5x^3+24p^4x^3+48p^3x^4+24p^3x^3\\
			&\quad+36p^2x^4+\frac{15}{4}p^6+21p^6x^2+12p^6x^4.
		\end{aligned}
		\]
		Using \(0\le p,x\le 1/4\), we obtain
		\[
		G_{1/2}(p,x)\le
		60-120p^2-48x^2+90px+\frac{11265}{512}p^2+10x^2.
		\]
		Thus
		\[
		G_{1/2}(p,x)\le 60-97p^2-38x^2+90px.
		\]
		Using
		\[
		90px\le 60p^2+\frac{135}{4}x^2,
		\]
		we obtain
		\[
		G_{1/2}(p,x)\le 60-37p^2-\frac{17}{4}x^2\le 60
		\qquad\text{on }Q_{11}.
		\]
		Combining all regions,
		\[
		G_0(p,x)\le 60,\qquad G_{1/2}(p,x)\le 60
		\]
		for all \((p,x)\in[0, 1]\times[0, 1]\). Since
		\[
		\max_{0\le t\le 1/2}G_{1}(p,x,t)=\max\left\{G_0(p,x),G_{1/2}(p,x)\right\},
		\]
		it follows that
		\[
		G_{1}(p,x,t)\le 60
		\qquad
		\text{for all }(p,x,t)\in [0,1]\times[0,1]\times\bigg[0,\frac12\bigg].
		\]
		Finally,
		\[
		G_{1}(0,0,t)=60
		\qquad \text{for every }0\le t\le \frac12.
		\]
		Hence, the maximum value is
		\[
		\max G_1(p_1,x,t)=60,
		\]
		and it is attained at
		$
		p_1=0,\,  x=0,\, 0\le t\le 1/2.
		$\\[2mm]
		Next, we deal with the case when $y=0,$ and $G_{2}(p_1,x, t):=H_{1}(p_1, x, 0, t)$
		\begin{align*}
			G_2(p_1,x,t):=&\,p_1^6(1-6t+21t^2+4t^3)
			+6p_1^4(1-p_1^2)(1-3t+10t^2)x \\[2mm]
			&+3p_1^2(1-p_1^2)\bigl(7-3p_1^2-4(2+p_1^2)t+8(7-p_1^2)t^2\bigr)x^2 \\[2mm]
			&+4(1-p_1^2)\bigl(1+7p_1^2+p_1^4+6(3-2p_1^2+2p_1^4)t\bigr)x^3 +12p_1^2(-1+p_1^2)^2x^4 \\[2mm]
			&+36(1-p_1^2)\bigl(p_1^2(1-2t)+2(1-p_1^2)x\bigr)(1-x^2) \\[2mm]
			&+12p_1(1-p_1^2)(1-x^2)\bigl(p_1^2(1-3t+12t^2)+3(1+p_1^2)x+4(2+p_1^2)tx\\[2mm]
			&+2(1-p_1^2)x^2\bigr).
		\end{align*}
		In this case we use an upper bound. We use the Bernstein method first on the whole region $(p_1, x, t)\in [0, 1]\times[0, 1]\times[0, 1/2].$
		Taking 
		$
		p=p_1,\, u=2t.
		$
		Since $0\le t\le 1/2,$ we have \(0\le u\le 1\). Define
		\begin{align*}
			H(p,x,u):=G_2\!\left(p,x,\frac{u}{2}\right).
		\end{align*}
		Then $H$ is a polynomial on the unit cube $[0, 1]\times[0, 1]\times[0, 1]$, of tri-degree $(6, 4, 3)$ in $(p, x, u)$.\\[2mm]
		Its full expansion is written as:
		\[
		\begin{aligned}
			H(p,x,u)&=\frac12 p^6u^3+6p^6u^2x^2-15p^6u^2x+\frac{21}{4}p^6u^2
			-24p^6ux^3+6p^6ux^2+9p^6ux\\
			&\quad-3p^6u+12p^6x^4-4p^6x^3+9p^6x^2-6p^6x+p^6
			+36p^5u^2x^2-36p^5u^2\\
			&\quad+24p^5ux^3-18p^5ux^2 -24p^5ux+18p^5u-24p^5x^4+36p^5x^3+36p^5x^2\\
			&\quad-36p^5x-12p^5 -48p^4u^2x^2+15p^4u^2x+48p^4ux^3-30p^4ux^2-9p^4ux\\
			&\quad+36p^4u -24p^4x^4-96p^4x^3+6p^4x^2+78p^4x-36p^4
			-36p^3u^2x^2\\
			&\quad+36p^3u^2+24p^3ux^3+18p^3ux^2-24p^3ux-18p^3u+48p^3x^4-60p^3x^2\\
			&\quad+12p^3
			+42p^2u^2x^2-60p^2ux^3 +24p^2ux^2-36p^2u+12p^2x^4+168p^2x^3\\
			&\quad-15p^2x^2-144p^2x+36p^2-48pux^3+48pux-24px^4-36px^3+24px^2\\
			&\quad+36px+36ux^3-68x^3+72x.
		\end{aligned}
		\]
		Now write $H$ in the  Bernstein basis:
		\[
		H(p,x,u)=\sum_{i=0}^{6}\sum_{j=0}^{4}\sum_{k=0}^{3}
		b_{ijk}\,\beta_i^{(6)}(p)\beta_j^{(4)}(x)\beta_k^{(3)}(u),
		\]
		where,
		\[
		\beta_r^{(n)}(s)=\binom{n}{r}s^r(1-s)^{n-r}.
		\]
		For each fixed $k=0, 1, 2, 3$, let
		\[
		B^{(k)}=(b_{ijk})_{0\le i\le 6,\;0\le j\le 4}.
		\]
		Rows correspond to $i=0,\dots,6$, and columns correspond to $j=0,\dots,4$.
		The Bernstein coefficients are as follows.\\[2mm]
		For $k=0$,
		\[
		B^{(0)}=
		\begin{pmatrix}
			0 & 18 & 36 & 37 & 4\\[3mm]
			0 & \frac{39}{2} & \frac{119}{3} & 42 & 4\\[3mm]
			\frac{12}{5} & 21 & \frac{1223}{30} & \frac{89}{2} & \frac{39}{5}\\[3mm]
			\frac{39}{5} & \frac{231}{10} & \frac{197}{5} & \frac{218}{5} & \frac{77}{5}\\[3mm]
			\frac{72}{5} & \frac{253}{10} & \frac{539}{15} & \frac{77}{2} & 22\\[3mm]
			16 & \frac{45}{2} & 27 & \frac{53}{2} & 18\\[3mm]
			1 & 1 & 1 & 1 & 1
		\end{pmatrix}.
		\]
		For $k=1$,
		\[
		B^{(1)}=
		\begin{pmatrix}
			0 & 18 & 36 & 40 & 16\\[3mm]
			0 & \frac{121}{6} & 41 & \frac{139}{3} & 16\\[3mm]
			\frac{8}{5} & \frac{323}{15} & \frac{769}{18} & \frac{493}{10} & \frac{91}{5}\\[3mm]
			\frac{51}{10} & \frac{223}{10} & \frac{2449}{60} & \frac{953}{20} & \frac{113}{5}\\[3mm]
			\frac{46}{5} & \frac{1339}{60} & \frac{3221}{90} & \frac{2449}{60} & \frac{127}{5}\\[3mm]
			10 & \frac{73}{4} & \frac{151}{6} & \frac{107}{4} & 19\\[3mm]
			0 & 0 & 0 & 0 & 0
		\end{pmatrix}.
		\]
		For $k=2$,
		\[
		B^{(2)}=
		\begin{pmatrix}
			0 & 18 & 36 & 43 & 28\\[3mm]
			0 & \frac{125}{6} & \frac{127}{3} & \frac{152}{3} & 28\\[3mm]
			\frac{4}{5} & \frac{331}{15} & \frac{269}{6} & \frac{1637}{30} & \frac{443}{15}\\[3mm]
			3 & \frac{221}{10} & \frac{216}{5} & \frac{267}{5} & \frac{163}{5}\\[3mm]
			\frac{32}{5} & \frac{1309}{60} & \frac{1157}{30} & \frac{937}{20} & \frac{101}{3}\\[3mm]
			8 & \frac{221}{12} & \frac{169}{6} & \frac{129}{4} & \frac{77}{3}\\[3mm]
			\frac34 & \frac34 & \frac34 & \frac34 & \frac34
		\end{pmatrix}.
		\]
		For $k=3$,\\
		\[
		B^{(3)}=
		\begin{pmatrix}
			0 & 18 & 36 & 46 & 40\\[3mm]
			0 & \frac{43}{2} & \frac{131}{3} & 55 & 40\\[3mm]
			0 & \frac{113}{5} & \frac{471}{10} & \frac{603}{10} & \frac{209}{5}\\[3mm]
			\frac32 & \frac{45}{2} & \frac{931}{20} & \frac{1217}{20} & \frac{227}{5}\\[3mm]
			6 & \frac{119}{5} & \frac{664}{15} & \frac{283}{5} & \frac{234}{5}\\[3mm]
			10 & 23 & 36 & 43 & 38\\[3mm]
			\frac{15}{4} & \frac{15}{4} & \frac{15}{4} & \frac{15}{4} & \frac{15}{4}
		\end{pmatrix}.
		\]\\
		By the Bernstein principle on $[0, 1]\times[0, 1]\times[0, 1]$, we have 
		\[
		H(p,x,u)\le \max_{0\le i\le 6,\;0\le j\le 4,\;0\le k\le 3} b_{ijk}.
		\]
		From the matrices above, the largest Bernstein coefficient is
		$
		{1217}/{20},
		$
		attained at the entry $(i,j,k)=(3,3,3)$.
		Therefore,
		\[
		H(p,x,u)\le \frac{1217}{20}
		\qquad\text{for all }(p,x,u)\in[0,1]\times[0, 1]\times[0, 1].
		\]
		Returning to $u=2t$, we obtain
		\[
		G_2(p_1,x,t)\le \frac{1217}{20}
		\qquad\text{for all }0\le p_1\le 1,\ 0\le x\le 1,\ 0\le t\le \frac12.
		\]
		Hence  Bernstein upper bound after one-step is
		\[
		\max G_{2}(p_1, x, t)\le \frac{1217}{20}=60.85.
		\]
		We now use the Bernstein method one more time to obtain a better upper bound by subdividing the $(p, x)$-region into four sub-regions.\\[1mm]
		Let
		\[
		u=2t,\qquad 0\le u\le 1,
		\]
		and keep
		\[
		H(p,x,u):=G_2\!\left(p,x,\frac{u}{2}\right),
		\qquad p=p_1.
		\]
		From the previous step, $H$ is a polynomial of tri-degree $(6, 4, 3)$ on $[0, 1]\times[0, 1]\times[0, 1]$, and the one-shot Bernstein bound gave
		\[
		H(p,x,u)\le \frac{1217}{20}=60.85.
		\]
		Now divide the $(p, x)$-square into four rectangles:
		\begin{align*}
			Q_1=\left[0,\frac12\right]\times\left[0,\frac12\right],\qquad
			Q_2=\left[0,\frac12\right]\times\left[\frac12,1\right],
			\\[3mm]
			Q_3=\left[\frac12,1\right]\times\left[0,\frac12\right],\qquad
			Q_4=\left[\frac12,1\right]\times\left[\frac12,1\right].
		\end{align*}
		On each $Q_{r}$, keep $u \in [0, 1]$, affine-reparametrize back to the unit cube, and then compute the Bernstein coefficients of tri-degree $(6, 4, 3)$.\\[2mm]
		For $Q_{1}$, set
		\[
		p=\frac{P}{2},\qquad x=\frac{X}{2},\qquad 0\le P,X,U\le 1.
		\]
		Then the transformed polynomial
		\[
		H\!\left(\frac{P}{2},\frac{X}{2},U\right)
		\]
		is written in the  Bernstein basis as
		\[
		H\!\left(\frac{P}{2},\frac{X}{2},U\right)
		=
		\sum_{i=0}^{6}\sum_{j=0}^{4}\sum_{k=0}^{3}
		b^{(Q_1)}_{ijk}\,
		\beta_i^{(6)}(P)\beta_j^{(4)}(X)\beta_k^{(3)}(U),
		\]
		where
		\[
		\beta_r^{(n)}(s)=\binom{n}{r}s^r(1-s)^{n-r}.
		\]
		For each fixed $k=0,1,2,3$, let
		\[
		B_{Q_1}^{(k)}=\bigl(b^{(Q_1)}_{ijk}\bigr)_{0\le i\le 6,\;0\le j\le 4}.
		\]
		For $k=0$,\\
		\[
		B_{Q_1}^{(0)}=
		\begin{pmatrix}
			0 & 9 & 18 & \frac{199}{8} & \frac{55}{2}\\[3mm]
			0 & \frac{75}{8} & \frac{113}{6} & \frac{837}{32} & 29\\[3mm]
			\frac{3}{5} & \frac{201}{20} & \frac{629}{32} & \frac{4351}{160} & \frac{151}{5}\\[3mm]
			\frac{15}{8} & \frac{111}{10} & \frac{6569}{320} & \frac{1793}{64} & \frac{311}{10}\\[3mm]
			\frac{15}{4} & \frac{3997}{320} & \frac{6861}{320} & \frac{2289}{80} & \frac{2533}{80}\\[3mm]
			\frac{95}{16} & \frac{1791}{128} & \frac{8531}{384} & \frac{14763}{512} & \frac{1015}{32}\\[3mm]
			\frac{505}{64} & \frac{3865}{256} & \frac{11471}{512} & \frac{905}{32} & \frac{3953}{128}
		\end{pmatrix}.
		\]\\
		For $k=1$,\\
		\[
		B_{Q_1}^{(1)}=
		\begin{pmatrix}
			0 & 9 & 18 & \frac{101}{4} & 29\\[3mm]
			0 & \frac{229}{24} & \frac{115}{6} & \frac{2591}{96} & 31\\[3mm]
			\frac{2}{5} & \frac{611}{60} & \frac{5797}{288} & \frac{679}{24} & \frac{3899}{120}\\[3mm]
			\frac{99}{80} & \frac{1753}{160} & \frac{40121}{1920} & \frac{9323}{320} & \frac{10697}{320}\\[3mm]
			\frac{49}{20} & \frac{4543}{384} & \frac{61801}{2880} & \frac{56641}{1920} & \frac{1013}{30}\\[3mm]
			\frac{123}{32} & \frac{405}{32} & \frac{5555}{256} & \frac{14995}{512} & \frac{4273}{128}\\[3mm]
			\frac{81}{16} & \frac{6705}{512} & \frac{10911}{512} & \frac{14445}{512} & \frac{4095}{128}
		\end{pmatrix}.
		\]\\
		For $k=2$,\\
		\[
		B_{Q_1}^{(2)}=
		\begin{pmatrix}
			0 & 9 & 18 & \frac{205}{8} & \frac{61}{2}\\[3mm]
			0 & \frac{233}{24} & \frac{39}{2} & \frac{2671}{96} & 33\\[3mm]
			\frac{1}{5} & \frac{619}{60} & \frac{9893}{480} & \frac{4707}{160} & \frac{4181}{120}\\[3mm]
			\frac{27}{40} & \frac{871}{80} & \frac{6837}{320} & \frac{973}{32} & \frac{2879}{80}\\[3mm]
			\frac{29}{20} & \frac{7343}{640} & \frac{873}{40} & \frac{11837}{384} & \frac{17491}{480}\\[3mm]
			\frac{39}{16} & \frac{577}{48} & \frac{8429}{384} & \frac{15683}{512} & \frac{3467}{96}\\[3mm]
			\frac{867}{256} & \frac{6285}{512} & \frac{11001}{512} & \frac{15135}{512} & \frac{8889}{256}
		\end{pmatrix}.
		\]\\
		For $k=3$,\\
		\[
		B_{Q_1}^{(3)}=
		\begin{pmatrix}
			0 & 9 & 18 & 26 & 32\\[3mm]
			0 & \frac{79}{8} & \frac{119}{6} & \frac{917}{32} & 35\\[3mm]
			0 & \frac{209}{20} & \frac{10129}{480} & \frac{1223}{40} & \frac{149}{4}\\[3mm]
			\frac{3}{16} & \frac{1743}{160} & \frac{14039}{640} & \frac{5093}{160} & \frac{12409}{320}\\[3mm]
			\frac{3}{4} & \frac{1827}{160} & \frac{10819}{480} & \frac{2607}{80} & \frac{6349}{160}\\[3mm]
			\frac{55}{32} & \frac{1547}{128} & \frac{17641}{768} & \frac{16827}{512} & \frac{5109}{128}\\[3mm]
			\frac{735}{256} & \frac{3237}{256} & \frac{11745}{512} & \frac{8277}{256} & \frac{10005}{256}
		\end{pmatrix}.
		\]\\
		Hence the largest Bernstein coefficient on $Q_1$ is
		\[
		\max B_{Q_1}=\frac{5109}{128}\approx 39.9140625,
		\]
		attained at the entry $(i,j,k)=(5,4,3)$.\\[2mm]
		For $Q_2$, set
		\[
		p=\frac{P}{2},\qquad x=\frac{1+X}{2},\qquad 0\le P,X,U\le 1.
		\]
		Then the transformed polynomial
		\[
		H\!\left(\frac{P}{2},\frac{1+X}{2},U\right)
		\]
		is written in the  Bernstein basis as
		\[
		H\!\left(\frac{P}{2},\frac{1+X}{2},U\right)
		=
		\sum_{i=0}^{6}\sum_{j=0}^{4}\sum_{k=0}^{3}
		b^{(Q_2)}_{ijk}\,
		\beta_i^{(6)}(P)\beta_j^{(4)}(X)\beta_k^{(3)}(U),
		\]
		where,
		\[
		\beta_r^{(n)}(s)=\binom{n}{r}s^r(1-s)^{n-r}.
		\]
		For each fixed $k=0,1,2,3$, let
		\[
		B_{Q_2}^{(k)}=\bigl(b^{(Q_2)}_{ijk}\bigr)_{0\le i\le 6,\;0\le j\le 4}.
		\]
		For $k=0$,
		\[
		B_{Q_2}^{(0)}=
		\begin{pmatrix}
			\frac{55}{2} & \frac{241}{8} & \frac{57}{2} & \frac{41}{2} & 4\\[3mm]
			29 & \frac{1019}{32} & \frac{725}{24} & \frac{87}{4} & 4\\[3mm]
			\frac{151}{5} & \frac{5313}{160} & \frac{5069}{160} & \frac{1853}{80} & \frac{99}{20}\\[3mm]
			\frac{311}{10} & \frac{10939}{320} & \frac{10517}{320} & \frac{3949}{160} & \frac{137}{20}\\[3mm]
			\frac{2533}{80} & \frac{2777}{80} & \frac{2153}{64} & \frac{8353}{320} & \frac{47}{5}\\[3mm]
			\frac{1015}{32} & \frac{17717}{512} & \frac{6481}{192} & \frac{1731}{64} & 12\\[3mm]
			\frac{3953}{128} & \frac{2143}{64} & \frac{16799}{512} & \frac{1729}{64} & \frac{223}{16}
		\end{pmatrix}.
		\]\\
		For $k=1$,
		\[
		B_{Q_2}^{(1)}=
		\begin{pmatrix}
			29 & \frac{131}{4} & 33 & 28 & 16\\[3mm]
			31 & \frac{3361}{96} & \frac{845}{24} & \frac{355}{12} & 16\\[3mm]
			\frac{3899}{120} & \frac{4403}{120} & \frac{53177}{1440} & \frac{1489}{48} & \frac{331}{20}\\[3mm]
			\frac{10697}{320} & \frac{12071}{320} & \frac{73097}{1920} & \frac{2063}{64} & \frac{353}{20}\\[3mm]
			\frac{1013}{30} & \frac{24341}{640} & \frac{110947}{2880} & \frac{21151}{640} & \frac{305}{16}\\[3mm]
			\frac{4273}{128} & \frac{19189}{512} & \frac{9749}{256} & \frac{1061}{32} & \frac{325}{16}\\[3mm]
			\frac{4095}{128} & \frac{18315}{512} & \frac{18651}{512} & \frac{16461}{512} & \frac{333}{16}
		\end{pmatrix}.
		\]\\
		For $k=2$,
		\[
		B_{Q_2}^{(2)}=
		\begin{pmatrix}
			\frac{61}{2} & \frac{283}{8} & \frac{75}{2} & \frac{71}{2} & 28\\[3mm]
			33 & \frac{3665}{96} & \frac{965}{24} & \frac{449}{12} & 28\\[3mm]
			\frac{4181}{120} & \frac{19327}{480} & \frac{4061}{96} & \frac{9373}{240} & \frac{1703}{60}\\[3mm]
			\frac{2879}{80} & \frac{6651}{160} & \frac{13981}{320} & \frac{6453}{160} & \frac{583}{20}\\[3mm]
			\frac{17491}{480} & \frac{80743}{1920} & \frac{1417}{32} & \frac{78887}{1920} & \frac{7219}{240}\\[3mm]
			\frac{3467}{96} & \frac{63895}{1536} & \frac{4213}{96} & \frac{15755}{384} & \frac{1475}{48}\\[3mm]
			\frac{8889}{256} & \frac{20421}{512} & \frac{21573}{512} & \frac{20325}{512} & \frac{7803}{256}
		\end{pmatrix}.
		\]\\
		For $k=3$,
		\[
		B_{Q_2}^{(3)}=
		\begin{pmatrix}
			32 & 38 & 42 & 43 & 40\\[3mm]
			35 & \frac{1323}{32} & \frac{1085}{24} & \frac{181}{4} & 40\\[3mm]
			\frac{149}{4} & \frac{1757}{40} & \frac{4589}{96} & \frac{3781}{80} & \frac{809}{20}\\[3mm]
			\frac{12409}{320} & \frac{1829}{40} & \frac{31823}{640} & \frac{15671}{320} & \frac{827}{20}\\[3mm]
			\frac{6349}{160} & \frac{1871}{40} & \frac{24439}{480} & \frac{1607}{32} & \frac{849}{20}\\[3mm]
			\frac{5109}{128} & \frac{24045}{512} & \frac{39295}{768} & \frac{6485}{128} & \frac{173}{4}\\[3mm]
			\frac{10005}{256} & \frac{11733}{256} & \frac{25569}{512} & \frac{6357}{128} & \frac{10995}{256}
		\end{pmatrix}.
		\]\\
		Hence the largest Bernstein coefficient on $Q_2$ is
		\[
		\max B_{Q_2}=\frac{39295}{768}\approx 51.1653646,
		\]
		attained at the entry $(i,j,k)=(5,2,3)$.\\[2mm]
		For $Q_3$, set
		\[
		p=\frac{1+P}{2},\qquad x=\frac{X}{2},\qquad 0\le P,X,U\le 1.
		\]
		Then the transformed polynomial
		\[
		H\!\left(\frac{1+P}{2},\frac{X}{2},U\right)
		\]
		is written in the  Bernstein basis as
		\[
		H\!\left(\frac{1+P}{2},\frac{X}{2},U\right)
		=
		\sum_{i=0}^{6}\sum_{j=0}^{4}\sum_{k=0}^{3}
		b^{(Q_3)}_{ijk}\,
		\beta_i^{(6)}(P)\beta_j^{(4)}(X)\beta_k^{(3)}(U),
		\]
		where
		\[
		\beta_r^{(n)}(s)=\binom{n}{r}s^r(1-s)^{n-r}.
		\]
		For each fixed $k=0,1,2,3$, let
		\[
		B_{Q_3}^{(k)}=\bigl(b^{(Q_3)}_{ijk}\bigr)_{0\le i\le 6,\;0\le j\le 4}.
		\]
		For $k=0$,\\
		\[
		B_{Q_3}^{(0)}=
		\begin{pmatrix}
			\frac{505}{64} & \frac{3865}{256} & \frac{11471}{512} & \frac{905}{32} & \frac{3953}{128}\\[3mm]
			\frac{315}{32} & \frac{1037}{64} & \frac{17351}{768} & \frac{14197}{512} & \frac{1923}{64}\\[3mm]
			\frac{185}{16} & \frac{1353}{80} & \frac{42611}{1920} & \frac{16897}{640} & \frac{4531}{160}\\[3mm]
			\frac{25}{2} & \frac{535}{32} & \frac{3321}{160} & \frac{15297}{640} & \frac{405}{16}\\[3mm]
			\frac{237}{20} & \frac{1187}{80} & \frac{2107}{120} & \frac{3133}{160} & \frac{817}{40}\\[3mm]
			\frac{17}{2} & \frac{81}{8} & \frac{23}{2} & \frac{199}{16} & \frac{51}{4}\\[3mm]
			1 & 1 & 1 & 1 & 1
		\end{pmatrix}.
		\]\\
		For $k=1$,\\
		\[
		B_{Q_3}^{(1)}=
		\begin{pmatrix}
			\frac{81}{16} & \frac{6705}{512} & \frac{10911}{512} & \frac{14445}{512} & \frac{4095}{128}\\[3mm]
			\frac{201}{32} & \frac{3465}{256} & \frac{1339}{64} & \frac{13895}{512} & \frac{3917}{128}\\[3mm]
			\frac{293}{40} & \frac{2609}{192} & \frac{114647}{5760} & \frac{48391}{1920} & \frac{6769}{240}\\[3mm]
			\frac{627}{80} & \frac{2067}{160} & \frac{34463}{1920} & \frac{14161}{640} & \frac{1569}{64}\\[3mm]
			\frac{73}{10} & \frac{5281}{480} & \frac{10483}{720} & \frac{1675}{96} & \frac{2291}{120}\\[3mm]
			5 & \frac{113}{16} & \frac{215}{24} & \frac{167}{16} & \frac{45}{4}\\[3mm]
			0 & 0 & 0 & 0 & 0
		\end{pmatrix}.
		\]\\
		For $k=2$,\\
		\[
		B_{Q_3}^{(2)}=
		\begin{pmatrix}
			\frac{867}{256} & \frac{6285}{512} & \frac{11001}{512} & \frac{15135}{512} & \frac{8889}{256}\\[3mm]
			\frac{555}{128} & \frac{9623}{768} & \frac{16145}{768} & \frac{14587}{512} & \frac{12799}{384}\\[3mm]
			\frac{1679}{320} & \frac{1499}{120} & \frac{38339}{1920} & \frac{10193}{384} & \frac{9879}{320}\\[3mm]
			\frac{939}{160} & \frac{381}{32} & \frac{2891}{160} & \frac{3003}{128} & \frac{4329}{160}\\[3mm]
			\frac{463}{80} & \frac{1651}{160} & \frac{297}{20} & \frac{1501}{80} & \frac{1711}{80}\\[3mm]
			\frac{35}{8} & \frac{335}{48} & \frac{19}{2} & \frac{93}{8} & \frac{313}{24}\\[3mm]
			\frac{3}{4} & \frac{3}{4} & \frac{3}{4} & \frac{3}{4} & \frac{3}{4}
		\end{pmatrix}.
		\]\\
		For $k=3$,\\
		\[
		B_{Q_3}^{(3)}=
		\begin{pmatrix}
			\frac{735}{256} & \frac{3237}{256} & \frac{11745}{512} & \frac{8277}{256} & \frac{10005}{256}\\[3mm]
			\frac{515}{128} & \frac{845}{64} & \frac{8797}{384} & \frac{16281}{512} & \frac{153}{4}\\[3mm]
			\frac{343}{64} & \frac{4369}{320} & \frac{14347}{640} & \frac{19491}{640} & \frac{11633}{320}\\[3mm]
			\frac{213}{32} & \frac{2199}{160} & \frac{13553}{640} & \frac{17899}{640} & \frac{10559}{320}\\[3mm]
			\frac{119}{16} & \frac{1033}{80} & \frac{4453}{240} & \frac{473}{20} & \frac{439}{16}\\[3mm]
			\frac{55}{8} & \frac{81}{8} & \frac{107}{8} & \frac{65}{4} & \frac{147}{8}\\[3mm]
			\frac{15}{4} & \frac{15}{4} & \frac{15}{4} & \frac{15}{4} & \frac{15}{4}
		\end{pmatrix}.
		\]\\
		Hence the largest Bernstein coefficient on $Q_3$ is
		\[
		\max B_{Q_3}=\frac{10005}{256}\approx 39.08203125,
		\]
		attained at the entry $(i,j,k)=(0,4,3)$.\\[2mm]
		For $Q_4$, set
		\[
		p=\frac{1+P}{2},\qquad x=\frac{1+X}{2},\qquad 0\le P,X,U\le 1.
		\]
		Then the transformed polynomial
		\[
		H\!\left(\frac{1+P}{2},\frac{1+X}{2},U\right)
		\]
		is written in the  Bernstein basis as
		\[
		H\!\left(\frac{1+P}{2},\frac{1+X}{2},U\right)
		=
		\sum_{i=0}^{6}\sum_{j=0}^{4}\sum_{k=0}^{3}
		b^{(Q_4)}_{ijk}\,
		\beta_i^{(6)}(P)\beta_j^{(4)}(X)\beta_k^{(3)}(U),
		\]
		where,
		\[
		\beta_r^{(n)}(s)=\binom{n}{r}s^r(1-s)^{n-r}.
		\]
		For each fixed $k=0,1,2,3$, let
		\[
		B_{Q_4}^{(k)}=\bigl(b^{(Q_4)}_{ijk}\bigr)_{0\le i\le 6,\;0\le j\le 4}.
		\]
		For $k=0$,\\
		\[
		B_{Q_4}^{(0)}=
		\begin{pmatrix}
			\frac{3953}{128} & \frac{2143}{64} & \frac{16799}{512} & \frac{1729}{64} & \frac{223}{16}\\[3mm]
			\frac{1923}{64} & \frac{16571}{512} & \frac{24473}{768} & \frac{1727}{64} & \frac{127}{8}\\[3mm]
			\frac{4531}{160} & \frac{19351}{640} & \frac{11467}{384} & \frac{8313}{320} & \frac{343}{20}\\[3mm]
			\frac{405}{16} & \frac{17103}{640} & \frac{132}{5} & \frac{47}{2} & \frac{341}{20}\\[3mm]
			\frac{817}{40} & \frac{3403}{160} & \frac{314}{15} & \frac{303}{16} & \frac{59}{4}\\[3mm]
			\frac{51}{4} & \frac{209}{16} & \frac{51}{4} & \frac{93}{8} & \frac{19}{2}\\[3mm]
			1 & 1 & 1 & 1 & 1
		\end{pmatrix}.
		\]\\
		For $k=1$,\\
		\[
		B_{Q_4}^{(1)}=
		\begin{pmatrix}
			\frac{4095}{128} & \frac{18315}{512} & \frac{18651}{512} & \frac{16461}{512} & \frac{333}{16}\\[3mm]
			\frac{3917}{128} & \frac{17441}{512} & \frac{4451}{128} & \frac{7973}{256} & \frac{341}{16}\\[3mm]
			\frac{6769}{240} & \frac{19971}{640} & \frac{183779}{5760} & \frac{1161}{40} & \frac{337}{16}\\[3mm]
			\frac{1569}{64} & \frac{17219}{640} & \frac{52811}{1920} & \frac{8123}{320} & \frac{779}{40}\\[3mm]
			\frac{2291}{120} & \frac{9953}{480} & \frac{15217}{720} & \frac{9463}{480} & \frac{317}{20}\\[3mm]
			\frac{45}{4} & \frac{193}{16} & \frac{293}{24} & \frac{183}{16} & \frac{19}{2}\\[3mm]
			0 & 0 & 0 & 0 & 0
		\end{pmatrix}.
		\]\\
		For $k=2$,\\
		\[
		B_{Q_4}^{(2)}=
		\begin{pmatrix}
			\frac{8889}{256} & \frac{20421}{512} & \frac{21573}{512} & \frac{20325}{512} & \frac{7803}{256}\\[3mm]
			\frac{12799}{384} & \frac{58631}{1536} & \frac{31015}{768} & \frac{29465}{768} & \frac{11609}{384}\\[3mm]
			\frac{9879}{320} & \frac{67583}{1920} & \frac{14315}{384} & \frac{34331}{960} & \frac{9307}{320}\\[3mm]
			\frac{4329}{160} & \frac{19617}{640} & \frac{649}{20} & \frac{1257}{40} & \frac{4227}{160}\\[3mm]
			\frac{1711}{80} & \frac{1921}{80} & \frac{507}{20} & \frac{3957}{160} & \frac{343}{16}\\[3mm]
			\frac{313}{24} & \frac{347}{24} & \frac{91}{6} & \frac{713}{48} & \frac{317}{24}\\[3mm]
			\frac{3}{4} & \frac{3}{4} & \frac{3}{4} & \frac{3}{4} & \frac{3}{4}
		\end{pmatrix}.
		\]\\
		For $k=3$,\\
		\[
		B_{Q_4}^{(3)}=
		\begin{pmatrix}
			\frac{10005}{256} & \frac{11733}{256} & \frac{25569}{512} & \frac{6357}{128} & \frac{10995}{256}\\[3mm]
			\frac{153}{4} & \frac{22887}{512} & \frac{9353}{192} & \frac{6229}{128} & \frac{5459}{128}\\[3mm]
			\frac{11633}{320} & \frac{27041}{640} & \frac{29447}{640} & \frac{1479}{32} & \frac{13199}{320}\\[3mm]
			\frac{10559}{320} & \frac{24337}{640} & \frac{26429}{640} & \frac{13339}{320} & \frac{6071}{160}\\[3mm]
			\frac{439}{16} & \frac{1249}{40} & \frac{8089}{240} & \frac{2729}{80} & \frac{2531}{80}\\[3mm]
			\frac{147}{8} & \frac{41}{2} & \frac{175}{8} & \frac{177}{8} & \frac{167}{8}\\[3mm]
			\frac{15}{4} & \frac{15}{4} & \frac{15}{4} & \frac{15}{4} & \frac{15}{4}
		\end{pmatrix}.
		\]\\
		Hence the largest Bernstein coefficient on $Q_4$ is
		\[
		\max B_{Q_4}=\frac{25569}{512}\approx 49.939453125,
		\]
		attained at the entry $(i,j,k)=(0,2,3)$.
		Hence, by the Bernstein enclosure principle on each sub-regions,
		\[
		H(p,x,u)\le
		\max\left\{
		\frac{5109}{128},
		\frac{39295}{768},
		\frac{10005}{256},
		\frac{25569}{512}
		\right\}
		=
		\frac{39295}{768}.
		\]
		Therefore
		\[
		H(p,x,u)\le \frac{39295}{768}
		\quad\text{for all }(p,x,u)\in[0,1]\times[0, 1]\times[0, 1].
		\]
		Returning to $u=2t$, we obtain the improved upper bound
		\[
		G_2(p_1,x,t)\le \frac{39295}{768}
		\qquad\text{for all }0\le p_1\le 1,\ 0\le x\le 1,\ 0\le t\le 1/2.
		\]
		Thus, after one further subdivision into four sub-regions, the Bernstein upper bound improves from
		$
		{1217}/{20}=60.85
		$
		to
		$
		{39295}/{768}\approx 51.1653646.
		$ \\[2mm]
		Combining the cases of $y=0$ and $y=1,$ we obtain 
		\begin{align*}
			|H_{3}(1)| \le \frac{60}{8640}=\frac{1}{144}
		\end{align*}
		Next, we show that the extremal function for $H_{3}(1)$ is obtained by taking $w(z)=z^3,$ we have 
		\[
		1+\frac{z f''(z)}{f'(z)}=1+z^3+t z^6.
		\]
		Hence
		\begin{align*}
			\frac{f''(z)}{f'(z)}&=z^2+t z^5.
		\end{align*}
		Therefore
		\[
		\frac{d}{dz}\bigl(\log f'(z)\bigr)=z^2+t z^5.
		\]
		Integrating, we obtain
		\[
		\log f'(z)=\frac{z^3}{3}+\frac{t z^6}{6}+C.
		\]
		Using the normalization \(f'(0)=1\), we get \(C=0\). Thus
		\[
		f'(z)=\exp\!\left(\frac{z^3}{3}+\frac{t z^6}{6}\right).
		\]
		Integrating once again and using \(f(0)=0\), we obtain
		\[
		f(z)=\int_0^z \exp\!\left(\frac{\xi^3}{3}+\frac{t \xi^6}{6}\right)\,d\xi.
		\]
		Hence the extremal function is given by
		\[
		f(z)=\int_0^z \exp\!\left(\frac{\xi^3}{3}+\frac{t \xi^6}{6}\right)\,d\xi,
		\qquad t=\frac{m}{n}.
		\]
		Now
		\[
		\exp\!\left(\frac{z^3}{3}+\frac{t z^6}{6}\right)
		=
		\sum_{k,j\ge 0}\frac{1}{k!\,j!}
		\left(\frac{z^3}{3}\right)^k
		\left(\frac{t z^6}{6}\right)^j.
		\]
		Hence
		\[
		f(z)=
		\sum_{k,j\ge 0}
		\frac{t^j}{3^k 6^j k!j!}\,
		\frac{z^{3k+6j+1}}{3k+6j+1}.
		\]
		The expansion of $f$ is 
		\[
		f(z)=
		z+\frac{z^4}{12}
		+\frac{(1+3t)z^7}{126}
		+\frac{(1+9t)z^{10}}{1620}
		+\cdots.
		\]
		From the expansion, we obtain
		\[
		a_2=0,\qquad a_3=0,\qquad a_4=\frac{1}{12},\qquad a_5=0.
		\]
		Therefore
		\[
		H_3(1)=-\frac{1}{144},
		\]
		Thus, we have
		\[
		|H_3(1)|=\frac{1}{144}.
		\]
		This proves the sharpness of the result.
	\end{pf}
	
	\vspace{3mm}
	
	\noindent\textbf{Compliance of Ethical Standards:}\\
	
	\noindent\textbf{Conflict of interest.} The authors declare that there is no conflict  of interest regarding the publication of this paper.
	\vspace{1mm}
	
	\noindent\textbf{Data availability statement.}  Data sharing is not applicable to this article as no datasets were generated or analyzed during the current study.
	\vspace{1mm}
	
	\noindent\textbf{Authors contributions.} Both the authors have made equal contributions in reading, writing, and preparing the manuscript.\vspace{1mm}
	
	\noindent\textbf{Acknowledgment:} 
	The second named author acknowledges financial support from the Council of Scientific and Industrial Research (CSIR), Government of India, through a CSIR Fellowship.


\begin{thebibliography}{99}
		
		
		\bibitem{ChoiKimSugawa2007}
		\textsc{J. H. Choi}, \textsc{Y. C. Kim} and \textsc{T. Sugawa},
		A general approach to the Fekete--Szeg\H{o} problem,
		\textit{J. Math. Soc. Japan} \textbf{59}(3) (2007), 707--727.
		
		
		\bibitem{Goodman1983}
		\textsc{A. W. Goodman},
		\textit{Univalent Functions, Vol.~I},
		Mariner Publishing Co., Inc., Tampa, FL, 1983.
		
		
		\bibitem{KwonLeckoSim2018}
		\textsc{O. S. Kwon}, \textsc{A. Lecko} and \textsc{Y. J. Sim},
		On the fourth coefficient of functions in the Carath\'eodory class,
		\textit{Comput. Methods Funct. Theory} \textbf{18}(2) (2018), 307--314.
		
		\bibitem{LiberaZlotkiewicz1982}
		\textsc{R. J. Libera} and \textsc{E. J. Z{\l}otkiewicz},
		Early coefficients of the inverse of a regular convex function,
		\textit{Proc. Amer. Math. Soc.} \textbf{85}(2) (1982), 225--230.
		
		\bibitem{MaMinda1992}
		\textsc{W. C. Ma} and \textsc{D. Minda},
		A unified treatment of some special classes of univalent functions,
		in \textit{Proceedings of the Conference on Complex Analysis (Tianjin, 1992)}, 157--169.
		
		\bibitem{Pommerenke1966}
		\textsc{C. Pommerenke},
		On the coefficients and Hankel determinants of univalent functions,
		\textit{J. London Math. Soc.} \textbf{41} (1966), 111--122.
		
		\bibitem{ProkhorovSzynal1981}
		D.~V.~Prokhorov and J.~Szynal,
		Inverse coefficients for $(\alpha,\beta)$-convex functions,
		\textit{Annales Universitatis Mariae Curie-Sk{\l}odowska, Sectio A}
		\textbf{35} (1981), no.~15, 125--143.
		
	\end{thebibliography}
\end{document}